\documentclass[reqno, 11pt]{amsart}%
\usepackage{amssymb}
\usepackage[nesting]{hyperref}
\usepackage[pdftex]{graphicx}
\usepackage{listings}
\usepackage{multirow}
\usepackage{placeins}
\usepackage{color}
\usepackage{subfigure}
\usepackage{lscape}
\usepackage{dsfont}
\usepackage{amsmath}
\usepackage{amsfonts}
\usepackage{tikz}
\usepackage{verbatim}
\usepackage{accents}
\usepackage{bbm} %indicator 
\usepackage{natbib}
\usepackage[ruled]{algorithm2e}
\newcommand{\BlackBoxes}{\global\overfullrule5pt}

\BlackBoxes

\newcommand{\A}{\mathcal{A}}
\newcommand{\B}{\mathcal{B}}

\newcommand{\E}{\mathbb{E}}

\newcommand{\F}{\mathcal{F}}
\newcommand{\G}{\mathcal{G}}
\renewcommand{\H}{\mathcal{H}}

\newcommand{\MM}{\mathbb{M}}
\newcommand{\N}{\mathbb{N}}
\renewcommand{\P}{\mathbb{P}}

\newcommand{\R}{\mathbb{R}}
\newcommand{\T}{\mathcal{T}}

\newcommand{\X}{\mathcal{X}}

\newcommand{\Z}{\mathcal{Z}}
\newcommand{\Zf}{\mathfrak{Z}}
\newcommand{\1}{\mathbbm{1}}
\newcommand{\q}{F^{-1}}

\DeclareMathOperator{\dif}{d}
\DeclareMathOperator{\ES}{ES}
\DeclareMathOperator{\epi}{epi}
\DeclareMathOperator{\esssup}{ess\,sup}

\DeclareMathOperator{\id}{id}

\DeclareMathOperator{\VaR}{VaR}

\newcommand{\wh}{\widehat}

\newtheorem{theorem}{Theorem}
\newtheorem{corollary}[theorem]{Corollary}
\newtheorem{lemma}[theorem]{Lemma}
\newtheorem{proposition}[theorem]{Proposition}
\theoremstyle{definition}
\newtheorem{example}[theorem]{Example}
\newtheorem{remark}[theorem]{Remark}
\newtheorem{definition}[theorem]{Definition}
\newtheorem{assumption}[theorem]{Assumption}

\numberwithin{equation}{section} \numberwithin{theorem}{section}
\def\0{\kern0pt\-\nobreak\hskip0pt\relax}

\makeatletter
\AtBeginDocument{ \def\@serieslogo{ \vbox to\headheight{ \parindent\z@ \fontsize{6}{7\p@}\selectfont
\vss}}}

\def\makeoverbar#1#2#3#4#5#6#7{ \setbox0=\hbox{$\m@th#2\mkern#5mu{{}#3{}}\mkern#6mu$} \setbox1=\null \dimen@=#4\fontdimen8#13 \dimen@=3.5\dimen@
\advance\dimen@ by \ht0 \dimen@=-#7\dimen@ \advance\dimen@ by \wd0
\ht1=\ht0 \dp1=\dp0 \wd1=\dimen@
\dimen@=\fontdimen8#13 \fontdimen8#13=#4\fontdimen8#13
\rlap{\hbox to \wd0{$\m@th\hss#2{\overline{\box1}}\mkern#5mu$}}
\fontdimen8#13=\dimen@}
\def\mylabel#1#2{{\def\@currentlabel{#2}\label{#1}}}
\makeatother
\copyrightinfo{}

\begin{document}
\title[Minimizing  spectral risk measures applied to Markov Decision Processes]{Minimizing  spectral risk measures applied to Markov Decision Processes}
\author[N. \smash{B\"auerle}]{Nicole B\"auerle}
\address[N. B\"auerle]{Department of Mathematics,
Karlsruhe Institute of Technology (KIT), D-76128 Karlsruhe, Germany}

\email{\href{mailto:nicole.baeuerle@kit.edu}{nicole.baeuerle@kit.edu}}

\author[A. \smash{Glauner}]{Alexander Glauner}
\address[A. Glauner]{Department of Mathematics,
Karlsruhe Institute of Technology (KIT), D-76128 Karlsruhe, Germany}

\email{\href{mailto:alexander.glauner@kit.edu} {alexander.glauner@kit.edu}}

%\thanks{${}^*$ Department of Mathematics,
%Karlsruhe Institute of Technology (KIT), D-76128 Karlsruhe, Germany}
%\thanks{${}^\dagger$ Department of Mathematics, Karlsruhe Institute of Technology (KIT), D-76128 Karlsruhe, Germany}

\begin{abstract}
	We study the minimization of a spectral risk measure of the total discounted cost generated by a Markov Decision Process (MDP) over a finite or infinite planning horizon. The MDP is assumed to have Borel state and action spaces and the cost function may be unbounded above. The optimization problem is split into two minimization problems using an infimum representation for spectral risk measures. We show that the inner minimization problem can be solved as an ordinary MDP on an extended state space and give sufficient conditions under which an optimal policy exists. Regarding the infinite dimensional outer minimization problem, we prove the existence of a solution and derive an algorithm for its numerical approximation. Our results include the findings in \citet{BaeuerleOtt2011} in the special case that the risk measure is Expected Shortfall. As an application, we present a dynamic extension of the classical static optimal reinsurance problem, where an insurance company minimizes its cost of capital.
\end{abstract}
\maketitle

%\listoftodos

\makeatletter \providecommand\@dotsep{5} \makeatother
%\listoftodos[Changes in Orange/Red To Do List in Green / Blue]\relax

\vspace{0.5cm}
\begin{minipage}{14cm}
{\small
\begin{description}
\item[\rm \textsc{ Key words}]
{\small Risk-Sensitive Markov Decision Process; Spectral Risk Measure; Dynamic Reinsurance}
\item[\rm \textsc{AMS subject classifications}] 
{\small 90C40, 91G70, 91G05 }

\end{description}
}
\end{minipage}

%--------------------------------------------------------------------------------------------------
\section{Introduction}
%--------------------------------------------------------------------------------------------------
In the last decade, there have been various proposals to replace the expectation in the optimization of Markov Decision Processes (MDPs) by risk measures. The idea behind it is to take the risk-sensitivity of the decision maker into account. Using simply the expectation models a risk-neutral decision maker whose optimal policy sometimes can be very risky, for an example see e.g. \citet{BaeuerleOtt2011}. 

The literature can here be divided into two streams: Those papers which apply risk measures recursively and those which apply the risk measure to the total cost. The recursive approach for general MDP can for example be found in \citet{Ruszcynski2010,chu2014markov,BauerleGlauner2020}. The theory for these kind of models is rather different to the ones where the risk measures is applied to the total cost, since in the recursive approach we still get a recursive solution procedure directly. In this paper, we contribute to the second model class, i.e. we assume that a cost process is generated over discrete time by a decision maker and she aims at minimizing the risk measure applied to either the cost over a finite time horizon or over an infinite time horizon. The class of risk measures we consider here are so-called {\em spectral risk measures } which form a class of coherent risk measures including the {\em Expected Shortfall} or {\em Conditional Value-at-Risk}. More precisely spectral risk measures are mixtures of Expected Shortfall at different levels. 

For Expected Shortfall, the problem has already been treated in e.g. \citet{BaeuerleOtt2011,chow2015risk,ugurlu2017controlled}. Whereas in \citet{chow2015risk} the authors use a decomposition result of the Expected Shortfall shown in \citet{pflug2016time}, the authors of \citet{BaeuerleOtt2011} use the representation of Expected Shortfall as the solution of a global optimization problem over a real valued parameter, see \citet{RockafellarUryasev2000}. Interchanging the resulting two infima from the optimization problems yields a two-step method to solve the decision problem. Using the recent representation of spectral risk measures as an optimization problem over functions involving the convex conjugate in \citet{Pichler2015}, we follow a similar approach here. The problem can again be decomposed into an inner and outer optimization problem. The inner problem is to minimize the expected  convex function of the total cost. It can be solved with MDP techniques after a suitable extension of the original state space. Note that already here we get some difference to the  Expected Shortfall problem.  In contrast to the findings in \citet{BaeuerleOtt2011} who assume bounded cost or  \citet{ugurlu2017controlled}  who assume $L^1$ cost, we only require the cost to be bounded from below. No further integrability assumption is necessary here. Moreover, we allow for general Borel state and action spaces and give continuity and compactness conditions under which an optimal policy exists.  The major challenge is now the outer optimization problem, since we have to minimize over a function space and the dependence of the value function of the MDP on the functions is involved. However, we are again able to prove the existence of an optimal policy and an optimal function in the representation of the spectral risk measure. Moreover, by approximating the function space in the right way, we are able to reduce the outer optimization problem to a finite dimensional problem with a predetermined error bound. This yields an algorithm for the solution of the original optimization problem. Using an example from optimal reinsurance we show how our results can by applied.

Note that for the Expected Shortfall authors in \citet{chow2014algorithms,tamar2015policy} have developed gradient-based methods for the numerical computation of the optimal value and policy. For finite state and action spaces \citet{li2017quantile} provide an algorithm for quantile minimization of MDPs which is a similar problem. However, the outer optimization problem for spectral risk measures is much more demanding, since it is infinite dimensional. 

The paper is organized as follows: In the next section, we summarize definitions and properties of risk measures and introduce in particular the class of spectral risk measures which we consider here. In Section \ref{sec:decision_model}, we introduce the Markov Decision Model and give continuity and compactness assumptions which will later guarantee the existence of optimal policies. At the end of this section we formulate the spectral risk minimization problem of total cost. We also give some interpretations and show relations to other problems. In Section \ref{sec:finite}, we consider the inner problem with a finite time horizon. The necessary state space extension is explained as well as the recursive solution algorithm. Moreover, the existence of optimal policies is shown. In the next section we turn to the inner problem with infinite time horizon. We characterize the value function and show how optimal policies are obtained. In Section \ref{sec:monotone}, we discuss our model assumptions. In case the state space is the real line we show that the restrictive assumption of the continuity of the transition function which we need in the general model, can be replaced by  semicontinuity if some further monotonicity assumptions are satisfied. In Section \ref{sec:outer_existence}, we finally treat the outer optimization problem and prove the existence of an optimal function in the  representation of the spectral risk measure. The next section then deals with the numerical treatment of this problem. We show here that the infinite dimensional optimization problem can be approximated by a finite dimensional one. Last but not least in the final section we apply our findings to an optimal dynamic reinsurance problem. Problems of this type have been treated in a static setting before, see e.g. \citet{ChiTan2013,Cui2013,Lo2017,BauerleGlauner2018}, but we will consider them in a dynamic framework for the first time. The aim is to minimize the solvency capital calculated with a spectral risk measure by actively choosing reinsurance contracts for the next period. When the premium for the reinsurance contract is calculated by the expected premium principle we show that the optimal reinsurance contracts are of stop-loss type.

%--------------------------------------------------------------------------------------------------
\section{Spectral Risk Measures}\label{sec:spectral_risk_measures}
%--------------------------------------------------------------------------------------------------

Let $(\Omega, \A, \P)$ be a probability space and $L^0=L^0(\Omega, \A, \P)$ the vector space of real-valued  random variables. By $L^1$ we denote the subspace of integrable random variables and by $L^0_{\ge 0}$  the subspace  which consists of non-negative random variables. We follow the convention of the actuarial literature that positive realizations of random variables represent losses and negative ones gains. Let $\X \subseteq L^0$ be a convex cone. A \emph{risk measure} is a functional $\rho : \X \to \bar{\R}$. The following properties are relevant in this paper.

\begin{definition}\label{def:rm_properties}
	A risk measure $\rho : \X \to \bar{\R}$ is called
	\begin{enumerate}
		\item \emph{law-invariant} if $\rho(X)=\rho(Y)$ for $X,Y$ with the same distribution.
		\item \emph{monotone} if $X\leq Y$ implies $\rho(X) \leq \rho(Y)$.
		\item \emph{translation invariant} if $\rho(X+m)=\rho(X)+m$ for all $m \in \R \cap \X$.
		\item \emph{positive homogeneous} if $\rho(\lambda X)=\lambda\rho(X)$ for all $\lambda \in \R_+$.
		\item \emph{comonotonic additive} if $\rho(X+Y) = \rho(X)+\rho(Y)$ for all comonotonic $X,Y$.
		\item \emph{subadditive} if $\rho(X+Y)\leq \rho(X)+\rho(Y)$ for all $X,Y$.
	\end{enumerate}
\end{definition}

A risk measure is referred to as \emph{monetary} if it is monotone and translation invariant. It appears to be consensus in the literature that these two properties are a necessary minimal requirement for any risk measure. Monetary risk measures which are additionally positive homogeneous and subadditive are called \emph{coherent}. We will focus on the following class of risk measures. Here, $F_X(x)=\P(X\leq x), \ x \in \R$ denotes the distribution function and $\q_X(u)=\inf\{x \in \R: F_X(x)\geq u\}, \ u \in [0,1]$, the quantile function of a random variable $X$.

\begin{definition}\phantomsection\label{def:spectral_risk_measure}
	An increasing function $\phi:[0,1] \to \R_+$ with $\int_0^1 \phi(u) \dif u=1$ is called \emph{spectrum} and the functional $\rho_\phi: L^0_{\ge 0} \to \bar \R$ with
		\[ \rho_{\phi}(X)= \int_0^1 \q_X(u) \phi(u) \dif u \]
	is referred to as \emph{spectral risk measure}.
\end{definition}

Spectral risk measures were introduced by \citet{Acerbi2002}. They have all the properties listed in Definition \ref{def:rm_properties}. Properties a)-e) follow directly from respective properties of the quantile function. Verifying subadditivity is more involved, see \citet{Dhaene2000}. As part of the proof they showed that spectral risk measures preserve the increasing convex oder. Spectral risk measures belong to the larger class of distortion risk measures.

\begin{definition}
	An increasing right-continuous function $\varphi:[0,1] \to [0,1]$ with $\varphi(0)=0$ and $\varphi(1)=1$ is called \emph{distortion function} and the functional $\rho_\varphi: L^0_{\ge 0} \to \bar \R$ with
		\[ \rho_\varphi(X)= \int_0^1 \q_X(u)\dif \varphi (u) \]
	is referred to as \emph{distortion risk measure}.
\end{definition}

In the special case of a spectral risk measure, the distortion function is given by
\begin{align}\label{eq:distortion}
	\varphi(u)= \int_0^u \phi(s) \dif s, \qquad u \in [0,1]
\end{align}
and is convex. This also shows that it is no restriction to assume $\phi$ being right continuous (as the right derivative of a convex function). Conversely, for a convex distortion function without a jump in $1$, which implies continuity on $[0,1]$, one can always find a representation as in \eqref{eq:distortion} with $\phi$ being a spectrum. Consequently, all distortion risk measures with convex and continuous distortion function are spectral. It has been proven by \citet{Dhaene2000} that the convexity of $\varphi$ is equivalent to $\rho_\varphi$ being subadditive. 

Note that $\rho_\phi$ is finite on $L^1_{\geq 0}$ if the spectrum $\phi$ is bounded. On $L^0_{\ge 0}$ the value $+\infty$ is possible. \citet{Shapiro2013} has shown that a finite risk measure on $L^1_{\geq 0}$ with all the properties in Definition \ref{def:rm_properties} is already spectral with bounded spectrum.

\begin{example}
	The most widely used spectral risk measure is \emph{Expected Shortfall}
	\[ \ES_{\alpha}(X)= \frac{1}{1-\alpha}\int_{\alpha}^1 \q_X(u) \dif u, \qquad \alpha \in [0,1). \]
	Its spectrum $\phi(u)=\frac{1}{1-\alpha}\1_{[\alpha,1]}(u)$ is bounded. Especially in optimization, an infimum representation of Expected Shortfall going back to \citet{RockafellarUryasev2000} is very useful:
	\begin{align}\label{eq:ES_inf}
		\ES_{\alpha}(X) = \inf_{q \in \R} \left\lbrace q + \frac{1}{1-\alpha} \E[(X-q)^+] \right\rbrace, \qquad X \in L^1_{\geq 0}.
	\end{align}
	The infimum is attained at $q=\q_X(\alpha)$.
\end{example}

Henceforth, we assume w.l.o.g.\ that $\phi$ is right-continuous. Then $\nu([0,t])
=\phi(t)$ defines a Borel measure on $[0,1]$. Let us define a further measure $\mu$ by $\frac{d \mu}{d \nu}(\alpha)=(1-\alpha)$. Every spectral risk measure can be expressed as a mixture of Expected Shortfall over different confidence levels, see e.g.\ Proposition 8.18 in \citet{McNeil2015}.

\begin{proposition}\label{thm:ES_mix}
	 Let $\rho_{\phi}$ be a spectral risk measure. Then $\mu$ is a probability measure on $[0,1]$ and $\rho_{\phi}$ has the representation
	\[ \rho_{\phi}(X) = \int_0^1 \ES_{\alpha}(X) \mu(\dif \alpha), \qquad X\in L^0_{\ge 0} \]
\end{proposition}

When we allow to take the supremum on the r.h.s.\ over all probability measures $\mu$ we would get the superclass of coherent risk measures, see \citet{kusuoka}.

Using Proposition \ref{thm:ES_mix}, the infimum representation \eqref{eq:ES_inf} of Expected Shortfall can be generalized to spectral risk measures.

\begin{proposition}\label{thm:spectral_inf_rep}
	Let $\rho_{\phi}$ be a spectral risk measure with bounded spectrum. We denote by $G$ the set of increasing convex functions $g:\R \to \R$. Then it holds for $X \in L^0_{\geq 0}$
	\[ \rho_{\phi}(X) = \inf_{g \in G} \left\lbrace \E[g(X)] + \int_0^1 g^*(\phi(u)) \dif u\right\rbrace,  \]
	where $g^*$ is the convex conjugate of $g \in G$.
\end{proposition}
\begin{proof}
	For $X \in L^1_{\geq 0}$ the assertion has been proven by \citet{Pichler2015}. For non-integrable $X \in L^0_{\geq 0}$ it follows from Proposition \ref{thm:ES_mix}
	\[ \rho_{\phi}(X) = \int_0^1 \ES_{\alpha}(X) \mu(\dif \alpha) \geq \ES_{0}(X) =\E[X] =\infty. \]
	Now let $g \in G$ and $U_X \sim \mathcal{U}(0,1)$ be the generalized distributional transform of $X$, i.e.\ $\q_X(U_X)=X$ a.s. By the definition of the convex conjugate it holds $g(X) + g^*(\phi(U_X)) \geq X \phi(U_X)$. Hence, we have
	\begin{align*}
		\E[g(X)] + \E[g^*(\phi(U_X))]  \geq \E[X \, \phi(U_X)] = \E[\q_X(U_X)\, \phi(U_X)] = \rho_{\phi}(X) = \infty.
	\end{align*}
	Since $g \in G$ was arbitrary, the assertion follows.
\end{proof}

\begin{remark}\phantomsection\label{rem:minimizer}
		The proof by  \citet{Pichler2015} shows that for $X \in L^1_{\geq 0}$ the infimum is attained in  $g_{\phi,X}: \R \to \R$,
		$ g_{\phi,X}(x) = \int_0^1 \q_X(\alpha) + \frac{1}{1-\alpha}\left( x- \q_X(\alpha) \right)^+ \mu(\dif \alpha)$
		with $\mu$ from Proposition \ref{thm:ES_mix} and that the derivative of this function is $g_{\phi,X}'(x) =\phi(F_X(x))$ a.e.
\end{remark}

%--------------------------------------------------------------------------------------------------
\section{Markov Decision Model}\label{sec:decision_model}
%--------------------------------------------------------------------------------------------------

We consider the following standard Markov Decision Process with general Borel state and action space. The \emph{state space} $E$ is a Borel space with Borel $\sigma$-algebra $\B(E)$ and the \emph{action space} $A$ is a Borel space with Borel $\sigma$-Algebra $\B(A)$. The possible state-action combinations at time $n$ form a measurable subset $D_n$  of $E \times A$ such that $D_n$ contains the graph of a measurable mapping $E \to A$. The $x$-section of $D_n$,  
\[ D_n(x) = \{ a \in A: (x,a) \in D_n \}, \]
is the set of admissible actions in state $x \in E$ at time $n$. Note that the sets $D_n(x)$ are non-empty. We assume that the dynamics of the MDP are given by measurable \emph{transition functions} $T_n:D_n \times \Z \to E$ and depend on 
\emph{disturbances} $Z_1,Z_2,\dots$ which are independent random elements on a common probability space $(\Omega,\A,\P)$ with values in a measurable space $(\Z, \Zf)$. When the current state is $x_n$, the controller chooses action $a_n\in D_n(x_n)$ and $z_{n+1}$ is the realization of $Z_{n+1}$, then the next state is given by
\[ x_{n+1} = T_n(x_n,a_n,z_{n+1}). \] 
The \emph{one-stage cost function} $c_n:D_n\times E \to \R_+$ gives the cost $c_n(x,a,x')$  for choosing action $a$ if the system is in state $x$ at time $n$ and the next state is $x'$. The \emph{terminal cost function} $c_N: E \to \R_+$ gives the cost $c_N(x)$ if the system terminates in state $x$. Note that instead of non-negative cost we can equivalently consider cost which are bounded from below.

The model data is supposed to have the following continuity and compactness properties. 
\begin{assumption}\phantomsection\label{ass:continuity_compactness}
	\begin{enumerate}
		\item[(i)]  The sets $D_n(x)$ are compact and  $E \ni x \mapsto D_n(x)$ are upper semicontinuous, i.e. if $x_k\to x$ and $a_k\in D_n(x_k)$, $k\in\N$, then $(a_k)$ has an accumulation point in $D_n(x)$.
		\item[(ii)] The transition functions $T_n$ are continuous in $(x,a)$.
		\item[(iii)] The one-stage cost functions $c_n$ and the terminal cost function $c_N$ are lower semicontinuous. 
	\end{enumerate}
\end{assumption}

Under a finite planning horizon $N \in \N$, we consider the model data for $n=0,\dots,N-1$. The decision model is called \emph{stationary} if $D, T, c$ do not depend on $n$ and the disturbances are identically distributed.  If the model is stationary and the terminal cost is zero, we allow for an \emph{infinite time horizon} $N=\infty$.

For $n \in \N_0$ we denote by $\H_n$ the set of \emph{feasible histories} of the decision process up to time $n$
\begin{align*}
h_n = \begin{cases}
x_0, & \text{if } n=0,\\
(x_0,a_0,x_1, \dots, x_n), & \text{if } n \geq 1,
\end{cases}
\end{align*}
where $a_k \in D_k(x_k)$ for $k \in \N_0$. In order for the controller's decisions to be implementable, they must be based on the information available at the time of decision making, i.e.\ be functions of the history of the decision process. 
\begin{definition}
	\begin{enumerate}
		\item A measurable mapping $d_n: \mathcal{H}_n \to A$ with $d_n(h_n) \in D_n(x_n)$ for every $h_n \in \mathcal{H}_n$ is called  \emph{decision rule} at time $n$. A finite sequence $\pi=(d_0, \dots,d_{N-1})$ is called \emph{$N$-stage policy} and a sequence $\pi=(d_0, d_1, \dots)$ is called \emph{policy}.
		\item A decision rule at time $n$ is called \emph{Markov} if it  depends on the current state only, i.e.\ $d_n(h_n)=d_n(x_n)$ for all $h_n \in \mathcal{H}_n$. If all decision rules are Markov, the ($N$-stage) policy is called \emph{Markov}.
		\item An ($N$-stage) policy $\pi$ is called \emph{stationary} if $\pi=(d, \dots,d)$ or $\pi=(d,d,\dots)$, respectively, for some Markov decision rule $d$.
	\end{enumerate}
\end{definition}
With $\Pi$ and $\Pi^M$ we denote the sets of all policies and Markov policies respectively. It will be clear from the context if $N$-stage or infinite stage policies are meant. An admissible policy always exists as $D_n$ contains the graph of a measurable mapping.

Since risk measures are defined as real-valued mappings of random variables, we will work with a functional representation of the decision process. The law of motion does not need to be specified explicitly. We define for an initial state $x_0 \in E$ and a policy $\pi \in \Pi$
\begin{align*}
X^\pi_0=x_0, \qquad X^\pi_{n+1}= T(X_n^\pi,d_n(H_n^\pi),Z_{n+1}).
\end{align*}
Here, the process $(H_n^\pi)_{n \in \N_0}$ denotes the history of the decision process viewed as a random element, i.e.
\begin{align*}
H_0^\pi=x_0, \quad
H_1^\pi=\big(X_0^\pi,d_0(X_0^\pi),X_1^\pi\big),\quad \dots, \quad H_{n}^\pi=(H_{n-1}^\pi,d_{n-1}(H_{n-1}^\pi),X_{n}^\pi).
\end{align*}
Under a Markov policy the recourse on the random history of the decision process is not needed.

Even though the model is non-stationary we will explicitly introduce discounting by a factor $\beta >0$ since for the following state space extension it is relevant if there is discounting. Otherwise, stationary models with discounting would have to be treated separately. For a finite planning horizon $N \in \N$, the total discounted cost generated by a policy $\pi \in \Pi$ if the initial state is $x \in E$, is given by
\begin{align*}
	C_N^{\pi x} = \sum_{k=0}^{N-1} \beta^k c_k(X_k^\pi,d_k(H_k^\pi),X_{k+1}^\pi) + \beta^N c_N(X_N^\pi). 
\end{align*}
If the model is stationary and the planning horizon infinite, the total discounted cost is given by 
\begin{align*}
	C_\infty^{\pi x} = \sum_{k=0}^{\infty} \beta^k c(X_k^\pi,d_k(H_k^\pi),X_{k+1}^\pi). 
\end{align*}
For a generic total cost regardless of the planning horizon we write $C^{\pi x}$. Our aim is to find a policy $\pi \in \Pi$ which attains
\begin{align}
	\inf_{\pi \in \Pi} \rho_\phi(C_N^{\pi x}) \label{eq:opt_crit_finite}\\
	\intertext{or}
	\inf_{\pi \in \Pi} \rho_\phi(C_\infty^{\pi x}), \label{eq:opt_crit_infinite}
\end{align} 
respectively, for a fixed spectral risk measure $\rho_{\phi}: L^0_{\ge 0} \to \bar \R$ with $\phi(1)<\infty$, i.e.\ $\phi$ is bounded. We can apply Proposition \ref{thm:spectral_inf_rep} to reformulate the optimization problems \eqref{eq:opt_crit_finite} and \eqref{eq:opt_crit_infinite} to
\begin{align}\label{eq:outer_problem}
	\inf_{\pi \in \Pi} \rho_\phi \left( C^{\pi x} \right) &= \inf_{\pi \in \Pi} \inf_{g \in G} \left\lbrace \E[g(C^{\pi x})] + \int_0^1 g^*(\phi(u)) \dif u\right\rbrace \notag\\
	&= \inf_{g \in G}\inf_{\pi \in \Pi}  \left\lbrace \E[g(C^{\pi x})] + \int_0^1 g^*(\phi(u)) \dif u\right\rbrace \notag\\
	&= \inf_{g \in G} \left\lbrace \inf_{\pi \in \Pi} \E[g(C^{\pi x})] + \int_0^1 g^*(\phi(u)) \dif u\right\rbrace,
\end{align}
For fixed $g \in G$ we will refer to 
\begin{align}\label{eq:inner_problem}
	\inf_{\pi \in \Pi} \E[g(C^{\pi x})]
\end{align}
as \emph{inner optimization problem}. In the two following sections we, solve \eqref{eq:inner_problem} as an ordinary MDP on an extended state space. If $C^{\pi x} \in L^0_{\ge 0}$ but not in $L^1$, then $\rho_\phi(C^{\pi x})=\infty$. These policies are not interesting and can be excluded from the optimization.

Since an increasing convex function $g:\R \to \R$ can be viewed as a disutility function, optimality criterion \eqref{eq:inner_problem} implies that the expected disutility of the total discounted cost in minimized. If $g$ is strictly increasing, the optimization problem is not changed by applying $g^{-1}$, i.e.\ minimizing the corresponding certainty equivalent $g^{-1}\big(\E[g(C^{\pi x})]\big)$. For bounded one-stage cost functions, such problems are solved in \citet{BauerleRieder2014}. The special case of the exponential disutility function $g(x) = \exp(\gamma x), \ \gamma >0,$ has been studied first by \citet{HowardMatheson1972} in a decision model with finite state and action space. The term \emph{risk-sensitive MDP} goes back to them. The certainty equivalent corresponding to an exponential disutility is the entropic risk measure 
\[ \rho(X)= \frac{1}{\gamma} \log \E\left[ e^{\gamma X}  \right]. \]
It has been shown by \citet{Mueller2007} that an exponential disutility is the only case where the certainty equivalent defines a monetary risk measure apart from expectation itself (linear disutility).  

The concepts of spectral risk measures and expected disutilities (or corresponding certainty equivalents) can be combined to so-called \emph{rank-dependent expected disutilities} of the form $\rho_{\phi}(u(X))$, where $u$ is a disutility function. The corresponding certainty equivalent is $u^{-1}\big(\rho_{\phi}(u(X))\big)$. In fact, this concept works more generally for distortion risk measures and incorporates both expected disutilities (identity as distortion function) and distortion risk measures (identity as disutility function). The idea is that the expected disutility is calculated w.r.t.\ a distorted probability instead of the original probability measure. As long as the distorted probability is spectral, using a rank dependent disutility instead of $\rho_{\phi}$ leads to structurally the same inner problem as \eqref{eq:inner_problem}, only $g$ is replaced by $g(u(\cdot))$. Our results apply here, too. The certainty equivalent of a rank-dependent expected disutility combining an exponential disutility with a spectral risk measure is itself a convex (but not coherent) risk measure. It has been introduced by \citet{TsanakasDesli2003} as \emph{distortion-exponential risk measure}.

%--------------------------------------------------------------------------------------------------
\section{Inner Problem: Finite Planning Horizon}\label{sec:finite}
%--------------------------------------------------------------------------------------------------

Under a finite planning horizon $N \in \N$, we consider the non-stationary version of the decision model and our aim is to solve
\begin{align}\label{eq:inner_problem_finite}
	\inf_{\pi \in \Pi} \E[g(C_N^{\pi x})]
\end{align}
for an arbitrary but fixed increasing convex function $g \in G$.
We assume that for all $x\in E$ there is at least one policy $\pi$ s.t. $C_N^{\pi x}\in L^1$. Problem \eqref{eq:inner_problem_finite} is well-defined since the target function is bounded from below by $g(0)$. W.l.o.g.\ we assume $g\ge 0$. Note that the value $+\infty$ is possible. 

As the functions $g \in G$ are in general non-linear, the optimization problem cannot be solved with MDP techniques directly. This can be overcome by extending the state space to
\[ \wh E = E \times \R_+  \times (0,\infty) \]
with corresponding Borel $\sigma$-algebra following \citet{BauerleRieder2014}. A generic element of $\wh E$ is denoted by $(x,s,t)$. The idea is that $s$ summarizes the cost accumulated to far and that $t$ keeps track of the discounting. The action space $A$ and the admissible state-action combinations $D_n$, $n=0,\dots,N-1$ remain unchanged. Formally, one defines 
\[ \wh D_n = \{ (x,s,t,a) \in \wh E \times A: \ a \in D_n(x) \}, \qquad n=0,\dots,N-1 \]
implying $\wh D_n(x,s,t) = D_n(x),\ (x,s,t) \in \wh E$.  
The transition function on the new state space is given by $\wh T_n: \wh D_n \times \Z \to \wh E$, 
\[\wh T_n(x,s,t,a,z) = \begin{pmatrix}
	T_n(x,a,z)\\
	s+t c_n(x,a,T_n(x,a,z))\\
	\beta t
\end{pmatrix}, \qquad n=0,\dots,N-1. \] 
Feasible histories of the decision model with extended state space up to time $n$ have the form
\[ h_n = \begin{cases}
	(x_0,s_0,t_0), &  n=0,\\
	(x_0,s_0,t_0,a_0,x_1,s_1,t_1,a_1, \dots, x_n,s_n,t_n), & 
	n \geq 1,
\end{cases} \]
where $a_k \in \wh D_k(x_k,s_k,t_k)$, $k=0,\dots,N-1$, and the set of such histories is denoted by $\wh \H_n$. With $\wh \Pi$ and $\wh \Pi^M$ we denote the sets of history-dependent and Markov policies for the decision model with extended state space. We will write $\E_{n h_n}$ for a conditional expectation given $H_n^\pi=h_n, \ h_n \in \wh \H_n$. The value of a policy $\pi \in \wh\Pi$ at time $n=0,\dots,N$ is defined as 
\begin{align}\label{eq:policy_value}
	\begin{aligned}
		V_{N\pi}(h_N) &= g(s_N+ t_Nc_N(x_N)),\\
		V_{n\pi}(h_n) &= \E_{nh_n}\left[g\left(s_n + t_n\left(\sum_{k=n}^{N-1} \beta^{k-n} c_k(X_k^\pi,d_k(H_k^\pi),X_{k+1}^\pi) + \beta^{N-n} c_N(X_N^\pi)\right) \right)\right], 
	\end{aligned}
\end{align}
where $h_n \in \wh H_n$. The corresponding value functions are 
\begin{align}\label{eq:value_function}
	V_{n}(h_n) = \inf_{\pi \in \wh\Pi} V_{n\pi}(h_n), \qquad h_n \in \wh\H_n.
\end{align}
In the end, the quantity of interest is $V_0(x,0,1)$ which agrees with the infimal value of the original inner optimization problem \eqref{eq:inner_problem_finite}. But how do we get an optimal policy for problem \eqref{eq:inner_problem_finite}? When starting in $(x_0,0,1) \in \wh E$, a history $(x_0,a_0,x_1,a_1,\dots,x_N) \in \H_N$ of the original decision model uniquely determines the history $(x_0,s_0,t_0,a_0,x_1,s_1,t_1,a_1, \dots, x_N,s_N,t_N) \in \wh \H_N$ of the decision model with extended state space through
\[ s_n = \sum_{k=0}^{n-1} \beta^k c_k(x_k,a_k,x_{k+1}) \quad \text{and} \quad t_n=\beta^n, \qquad n=0,\dots,N. \]
Hence, for the initial state $(x_0,0,1) \in \wh E$, a Markov policy $\pi=(d_0,\dots,d_{N-1}) \in \wh \Pi^M$ with $d_n:\wh E \to \A$, which will turn out to be optimal for \eqref{eq:value_function}, can be perceived as a history-dependent policy $\pi'=(d_0',\dots,d_{N-1}') \in \Pi$ of the original decision model, since we can find measurable functions 
$d_n':\H_n \to A$ satisfying $d_n'(h_n) \in D_n(x_n)$ and
\[ d_n'(x_0,a_0,x_1,\dots,x_n) = d_n\left(x_n, \sum_{k=0}^{n-1} \beta^k c_k(x_k,a_k,x_{k+1}), \beta^{n}  \right). \]
Analogously, a history-dependent policy $\pi \in \wh \Pi$ can be regarded as a history-dependent policy of the original decision model. We can now proceed to deriving an iteration for the policy values \eqref{eq:policy_value}. 
 
\begin{proposition}\label{thm:value_iteration}
	The value of a policy $\pi \in \wh \Pi$ can be calculated recursively for $n=0,\dots,N-1$ and $h_n \in \wh\H_n$ as
	\begin{align*}
		V_{N\pi}(h_N) &= g(s_N+t_Nc_N(x_N))\\
		V_{n\pi}(h_n) &= \E\Big[V_{n+1\pi}\Big(h_n,\, d_n(h_n),\, \wh T_n(x_n,s_n,t_n,d_n(h_n),Z_{n+1})  \Big)\Big]\\
		&=\E\Big[V_{n+1\pi}\Big(h_n,\, d_n(h_n),\, T_n(x_n,d_n(h_n),Z_{n+1}),\\
		&\phantom{=\E\Big[V_{n+1\pi}\Big(}\  s_n + t_n c_n(x_n,d_n(h_n),T_n(x_n,d_n(h_n),Z_{n+1})),\, \beta t  \Big)\Big].
	\end{align*}
\end{proposition}
\begin{proof}
	The proof is by backward induction. At time $N$ there is nothing to show. Now assume the assertion holds for $n+1$, then the tower property of conditional expectation yields
	\begin{align*}
		V_{n\pi}(h_n) &= \E_{nh_n}\Bigg[g\Bigg(s_n + t_n\Bigg(\sum_{k=n}^{N-1} \beta^{k-n} c_k(X_k^\pi,d_k(H_k^\pi),X_{k+1}^\pi) + \beta^{N-n} c_N(X_N^\pi) \Bigg)\Bigg)\Bigg]\\
		&= \E_{nh_n}\Bigg[g\Bigg(s_n + t_n c_n(x_n,d_n(h_n),T_n(x_n,d_n(h_n),Z_{n+1}))\\
		&\phantom{= \E_{nh_n}\Bigg[g\Bigg(}\ + t_n \beta \Bigg(\sum_{k=n+1}^{N-1} \beta^{k-(n+1)} c_k(X_k^\pi,d_k(H_k^\pi),X_{k+1}^\pi) + \beta^{N-(n+1)} c_N(X_N^\pi) \Bigg)\Bigg)\Bigg]\\
		&= \E_{nh_n}\Bigg[ \E_{n+1(h_n,d_n(h_n),\wh T_n(x_n,s_n,t_n,d_n(h_n),Z_{n+1}))} \Bigg[\\
		&\phantom{= \E_{nh_n}\Bigg[\E}\ g\Bigg(s_n + t_n c_n(x_n,d_n(h_n),T_n(x_n,d_n(h_n),Z_{n+1}))\\
		&\phantom{= \E_{nh_n}\Bigg[\E g\Bigg(}\ +  t_n \beta \Bigg(\sum_{k=n+1}^{N-1} \beta^{k-(n+1)} c_k(X_k^\pi,d_k(H_k^\pi),X_{k+1}^\pi) + \beta^{N-(n+1)} c_N(X_N^\pi) \Bigg)\Bigg)\Bigg] \Bigg]\\
%		&=\E_{nh_n}\Bigg[ V_{n+1\pi}\Big(h_n,\,d_n(h_n),\, \wh %T_n(x_n,s_n,t_n,d_n(h_n),Z_{n+1})  \Big) \Bigg]\\
		&=\E\Bigg[ V_{n+1\pi}\Big(h_n,\,d_n(h_n),\, \wh T_n(x_n,s_n,t_n,d_n(h_n),Z_{n+1})  \Big) \Bigg]. \qedhere
	\end{align*} 
\end{proof}

\begin{remark}
	If there is no discounting or if the discounting is included in the non-stationary one-stage cost functions, the second summary variable $t$ is obviously not needed. In the special case that $\rho_{\phi}$ is the Expected Shortfall, one only has to consider the functions $g_q(x)= (x-q)^+, \ q \in \R$. Due to their positive homogeneity in $(x,q)$, it suffices to extend the state space by only one real-valued summary variable even if there is discounting, cf.\ \citet{BaeuerleOtt2011}. 
\end{remark}

Let us now consider specifically Markov policies $\pi \in \wh \Pi^M$. The function space
\begin{align*}
	\MM = \big\{ v: \wh E \to \R_+\mid \ &  v \text{ is lower semicontinuous,}\\
	& v(x,\cdot,\cdot) \text{ is  increasing for all } x \in E,\\
	& v(x,s,t) \geq g(s) \text{ for } (x,s,t) \in \wh E \big\}
\end{align*} 
turns out to be the set of potential value functions under such policies. In order to simplify the notation, we introduce the usual operators on $\MM$. All $v \in \MM$ are non-negative and thus at least quasi-integrable.

\begin{definition} For $v \in \MM$ and a Markov decision rule $d:\wh\E \to A$ we define
	\begin{align*}
		L_n v(x,s,t,a) &= \E\Big[ v\Big( \wh T_n(x,s,t,a,Z_{n+1})\Big)\Big]\\
		&=\E\Big[ v\Big(T_n(x,a,Z_{n+1}),\, s+tc_n(x,a,T_n(x,a,Z_{n+1})),\, \beta t\Big) \Big], &&(x,s,t,a) \in \wh D_n,\\
		\T_{nd} v(x,s,t) &= L_n v(x,s,t,d(x,s,t)), && (x,s,t) \in \wh E,\\
		\T_n v(x,s,t) &= \inf_{a \in D(x)} L_n v(x,s,t,a), && (x,s,t) \in \wh E.
	\end{align*}
\end{definition}

Note that the operators are monotone in $v$. Under a Markov policy $\pi=(d_0,\dots,d_{N-1})\in \wh \Pi^M$ the value iteration can be expressed with the operators. In order to distinguish from the history-depended case, we denote the Markov value functions with $J$. Setting $J_{N\pi}(x,s,t)=g(s + t c_N(x)), \ (x,s,t) \in \wh E$, we obtain for $n=0,\dots,N-1$ and $(x,s,t) \in \wh E$
\begin{align*}
	J_{n\pi}(x,s,t) &= \E\Big[J_{n+1\pi}\Big(T_n(x,d_n(x),Z_{n+1}),\, s + t c_n(x,d_n(x),T_n(x,d_n(x),Z_{n+1})),\, \beta t  \Big)\Big]\\
	&= \T_{nd_n} J_{n+1\pi}(x,s,t).
\end{align*}
The corresponding Markov value functions are defined for $n=0, \dots, N$ as
\begin{align*}
	J_{n}(x,s,t) = \inf_{\pi \in \Pi^M} J_{n\pi}(x,s,t) , \qquad (x,s,t) \in \wh E.
\end{align*}

The next result shows that $V_n$ satisfies a Bellman equation and proves that an optimal policy exists and is Markov.

\begin{theorem}\label{thm:finite}
Let Assumption \ref{ass:continuity_compactness} be satisfied. Then, for $n=0, \dots, N$ the value functions $V_n$ only depend on $(x_n,s_n,t_n)$, i.e.\ $V_n(h_n)=J_n(x_n,s_n,t_n)$ for all $h_n \in \wh \H_n$, lie in $\MM$ and satisfy the Bellman equation
	\begin{align*}
		J_N(x,s,t) &= g(s+tc_N(x)),\\
		J_n(x,s,t) &= \T_n J_{n+1}(x,s,t), \qquad (x,s,t) \in \wh E.
	\end{align*}
	Furthermore, for $n= 0, \dots, N-1$ there exist Markov decision rules $d_n^*:\wh E \to A$ with $\T_{nd_n^*} J_{n+1}=\T_{n} J_{n+1}$ and every sequence of such minimizers constitutes an optimal policy $\pi=(d_0^*,\dots,d_{N-1}^*) \in \wh \Pi^M$ for problem \eqref{eq:value_function}.
\end{theorem}

\begin{proof}
	The proof is by backward induction. At time $N$ we have $V_N(h_N)=J_N(x_N,s_N,t_N)=g(s_N+t_Nc_N(x_N)), \ h_N \in \wh \H_n$, which is 
	\begin{itemize}
		\item lower semicontinuous since $g$ is increasing and continuous (as a convex function on $\R$) and $c_N$ is lower semicontinuous,
		\item increasing in $(s_N,t_N)$ since $g$ is increasing  and $c_N$ is non-negative,
		\item bounded below by $g(s_N)$ since $g$ is increasing and $t_N c_N(x_N) \geq 0$,
	\end{itemize}
	i.e.\ in $\MM$. Assuming the assertion holds at time $n+1$ we have at time $n$ for $h_n \in \wh\H_n$
	\begin{align}
		V_n(h_n) &= \inf_{\pi \in \wh \Pi} V_{n\pi}(h_n)\notag\\
		&= \inf_{\pi \in \wh \Pi} \E\Big[V_{n+1\pi}\Big(h_n,\, d_n(h_n),\, \wh T_n(x_n,s_n,t_n,d_n(h_n),Z_{n+1})  \Big)\Big]\notag\\
		&\geq \inf_{\pi \in \wh \Pi} \E\Big[V_{n+1}\Big(h_n,\, d_n(h_n),\, \wh T_n(x_n,s_n,t_n,d_n(h_n),Z_{n+1})  \Big)\Big]\notag\\
		&= \inf_{\pi \in \wh \Pi} \E\Big[J_{n+1}\Big(\wh T_n(x_n,s_n,t_n,d_n(h_n),Z_{n+1})  \Big)\Big].\notag\\
		&= \inf_{a_n \in D_n(x_n)} \E\Big[J_{n+1}\Big(\wh T_n(x_n,s_n,t_n,a_n,Z_{n+1})  \Big)\Big].\label{eq:finite_proof_0}
	\end{align}
	The last equality holds since the minimization does not depend on the entire policy but only on $a_n=d_n(h_n)$. Here, objective and constraint depend on the history of the process only through $x_n$. Thus, given existence of a minimizing Markov decision rule $d_n^*$, \eqref{eq:finite_proof_0} equals $\T_{n d_n^*} J_{n+1}(x_n,s_n,t_n)$. Again by the induction hypothesis, there exists an optimal Markov policy $\pi^* \in \Pi^M$ such that $J_{n+1}=J_{n+1\pi^*}$. Hence, we have
	\begin{align*}
		V_n(h_n) \geq \T_{n d_n^*} J_{n+1\pi^*}(x_n,s_n,t_n) = J_{n\pi^*}(x_n,s_n,t_n) \geq J_{n}(x_n,s_n,t_n) \geq V_n(h_n).
	\end{align*}
	It remains to show the existence of a minimizing Markov decision rule $d_n^*$ and that $J_n \in \MM$. We want to apply Proposition 2.4.3 of \citet{BaeuerleRieder2011}. The set-valued mapping $\wh E \ni (x,s,t)\mapsto D_n(x)$ is compact-valued and upper semicontinuous. Next, we show that $\wh D_n \ni (x,s,t,a) \mapsto  L_n v (x,s,t,a)$ is lower semicontinuous for every $v \in \MM$. Let $\{(x_k,s_k,t_k,a_k)\}_{k \in \N}$ be a convergent sequence in $\wh D_n$ with limit $(x^*,s^*,t^*,a^*) \in \wh D_n$. The mapping 
	\[\wh D_n \ni (x,s,t,a) \mapsto v\Big(T_n(x,a,Z_{n+1}(\omega)),\, s+tc_n(x,a,T_n(x,a,Z_{n+1}(\omega))),\, \beta t\Big)   \]
	is lower semicontinuous. Since $v \geq g \geq 0$, we can apply Fatou's Lemma which yields 
	\begin{align*}
		&\liminf_{k \to \infty} L_nv(x_k,s_k,t_k,a_k)\\
		&= \liminf_{k \to \infty} \E\Big[ v\Big(T_n(x_k,a_k,Z_{n+1}),\, s_k+t_kc_n(x_k,a_k,T_n(x_k,a_k,Z_{n+1})),\, \beta t_k\Big) \Big]\\
		&\geq \E\Big[\liminf_{k \to \infty}  v\Big(T_n(x_k,a_k,Z_{n+1}),\, s_k+t_kc_n(x_k,a_k,T_n(x_k,a_k,Z_{n+1})),\, \beta t_k\Big) \Big]\\
		&\geq \E\Big[ v\Big(T_n(x^*,a^*,Z_{n+1}),\, s^*+t^*c_n(x^*,a^*,T_n(x^*,a^*,Z_{n+1})),\, \beta t^*\Big) \Big]\\
		&= L_nv(x^*,s^*,t^*,a^*).
	\end{align*}
	I.e.\ $L_n v$ is lower semicontinuous. With Proposition 2.4.3 of \citet{BaeuerleRieder2011} follows the existence of a minimizing decision rule $d_n^*$ and the lower semicontinuity of $\T_nv$.
	
	Now fix $x \in E$. The fact that $(s,t) \mapsto \T_n v(x,s,t)$ is increasing follows as in Theorem 2.4.14 in \citet{BaeuerleRieder2011}. The inequality $\T_n v(x,s,t) \geq g(s), \ (x,s,t) \in \wh E,$ is obvious. Taken together, we have $\T_n v \in \MM$ and the proof is complete.
\end{proof}

\begin{remark}
	From Theorem \ref{thm:finite} it follows that the sequence $\{(x_n,s_n,t_n)\}_{n=0}^{N-1}$ with
	\[ (s_n,t_n) = \left( \sum_{k=0}^{n-1} \beta^k c_k(x_k,a_k,x_{k+1}),\, \beta^n \right) \]
	is a \emph{sufficient statistic} of the decision model with the original state space in the sense of \citet{Hinderer1970}.
\end{remark}

%--------------------------------------------------------------------------------------------------
\section{Inner Problem: Infinite Planning Horizon}\label{sec:infinite}
%--------------------------------------------------------------------------------------------------

In this section, we consider the inner optimization problem \eqref{eq:inner_problem} of the risk-sensitive total cost minimization under an infinite planning horizon. This is reasonable if the terminal period is unknown or if one wants to approximate a model with a large but finite planning horizon. Solving the infinite horizon problem will turn out to be easier since it admits a stationary optimal policy. 

We study the stationary version of the decision model with no terminal cost, i.e.\ $D, T,c$ do not depend on $n$, $c_N\equiv 0$ and the disturbances are identically distributed. Let $Z$ be a representative of the disturbance distribution. Our aim is to solve
\begin{align}\label{eq:inner_problem_infinite}
	\inf_{\pi \in \Pi} \E[g(C_\infty^{\pi x})]
\end{align}
for an arbitrary but fixed increasing convex function $g \in G$. As in the previous section we assume w.l.o.g.\ that $g\ge 0$ and that for all $x\in E$ there exists a policy $\pi$ such that $C_\infty^{\pi x}\in L^1$.

The remarks in Section \ref{sec:finite} regarding connections to the minimization of (rank-dependent) expected disutilities and corresponding certainty equivalents apply in the infinite horizon case as well.

In order to obtain a value iteration, the state space is extended to $\wh E = E \times \R_+ \times (0,\infty)$ as in Section \ref{sec:finite}. The action space $A$ and the admissible state-action combinations $D$ remain unchanged, i.e.\
$\wh D = \{ (x,s,t,a) \in \wh E \times A: \ a \in D(x) \}$ and $\wh D(x,s,t) = D(x),\ (x,s,t) \in \wh E$. 
The transition function on the new state space is given by $\wh T: \wh D \times \Z \to \wh E$, 
\[\wh T(x,s,t,a,z) = \begin{pmatrix}
	T(x,a,z)\\
	s+t c(x,a,T(x,a,z))\\
	\beta t
\end{pmatrix}. \] 

Since the model with infinite planning horizon will be derived as a limit of the one with finite horizon, the consideration can be restricted to Markov policies $\pi=(d_1,d_2,\dots) \in \wh\Pi^M$ due to Theorem \ref{thm:finite}. For the relevant initial state $(x_0,0,1) \in \wh E$, a Markov policy $\pi \in \wh \Pi^M$ can be perceived as a history-dependent policy of the original decision model, cf.\ Section \ref{sec:finite}. When calculating limits, it is more convenient to index the value functions with the distance to the time horizon rather than the point in time. This is also referred to as \emph{forward form} of the value iteration and is only possible under Markov policies in a stationary model. There, the two ways of indexing are equivalent. The value of a policy $\pi=(d_0, d_1\dots ) \in \wh\Pi^M$ up to a planning horizon $N \in \N$ is
\begin{align*}
	J_{0\pi}(x,s,t) &= g(s) \notag \\
	J_{N\pi}(x,s,t) &= \E_{0x}\left[g\left(s + t \sum_{k=0}^{N-1} \beta^{k} c(X_k^\pi,d_k(X_k^\pi,\mathbf{s}_k^\pi,\mathbf{t}^\pi_k),X_{k+1}^\pi) \right)\right], 
\end{align*}
where $(X_n^\pi,\textbf{s}_n^\pi,\mathbf{t}^\pi_n)_{n \in \N}$ is the extended decision process under policy $\pi \in \wh\Pi^M$ with initial state $(x,s,t)\in \wh E$. The change of indexing makes it necessary to write the value iteration in terms of the \emph{shifted policy} $\vec \pi= (d_1,d_2, \dots)$ corresponding to $\pi=(d_0,d_1,\dots) \in \wh \Pi^M$:
\begin{align}\label{eq:infinite_value_iteration}
	J_{N\pi}(x,s,t) &= E\Big[J_{N-1\vec\pi}\Big(T(x,d_0(x,s,t),Z),\, s + t c(x,d_0(x,s,t),T(x,d_0(x,s,t),Z)),\, \beta t  \Big)\Big] \notag\\
	&= \T_{d_0} J_{N-1 \vec \pi}(x), \qquad (x,s,t) \in \wh E.
\end{align}
 The value function for finite planning horizon $N \in \N$ is given by
\begin{align*}
	J_N(x,s,t) = \inf_{\pi \in \wh \Pi^M} J_{N\pi}(x,s,t), \qquad (x,s,t) \in \wh E,
\end{align*}
and satisfies due to Theorem \ref{thm:finite} the Bellman equation
\begin{align*}
	J_N(x,s,t) =\T J_{N-1}(x,s,t) = \T^N g(x,s,t), \qquad (x,s,t) \in \wh E.
\end{align*}

The value of a policy $\pi \in \wh \Pi^M$ under an infinite planning horizon is defined as 
\begin{align*}
	J_{\infty \pi}(x,s,t) = \E_{0x}\left[g\left(s + t \sum_{k=0}^{\infty} \beta^{k} c(X_k^\pi,d_k(X_k^\pi,\mathbf{s}_k^\pi,\mathbf{t}^\pi_k),X_{k+1}^\pi) \right)\right], \qquad (x,s,t) \in \wh E.
\end{align*}
Note that $J_{\infty \pi}$ is well-defined since $c\ge 0$ and hence $J_{N\pi}$ is increasing. The infinite horizon value function is
\begin{align}\label{eq:opt_crit_inner_infinite_extended}
	J_{\infty}(x,s,t) = \inf_{\pi \in \Pi^M} J_{\infty \pi}(x,s,t), \qquad  (x,s,t) \in \wh E.
\end{align}

The  limit $J(x) = \lim_{N \to \infty} J_N(x), \ x \in E, $ which again exists since $J_N$ is increasing, is referred to as \emph{limit value function}. Note that $\MM$ is closed under pointwise convergence and hence $J\in \MM$.

\begin{theorem}\label{thm:infinite}
	Let Assumption 3.1 be satisfied. Then it holds:
	\begin{enumerate}
		\item The infinite horizon value function $J_\infty$ is the smallest fixed point of the Bellman operator $\T$ in $\MM$ and $J_\infty=J$. 
		\item There exists a Markov decision rule $d^*$ such that $\T_{d^*} J_\infty = \T J_\infty$ and each stationary policy $\pi^*=(d^*,d^*,\dots)$ induced by such a decision rule is optimal for optimization problem \eqref{eq:opt_crit_inner_infinite_extended}. 
	\end{enumerate}
\end{theorem}

\begin{proof}
	\begin{enumerate}
		\item First, we show that $J_\infty = J$. For all $N\in\N$ we have $J_{N\pi} \ge J_N$. Taking the limit $N\to\infty$ we obtain $J_{\infty \pi} \ge J$ for policies $\pi \in \wh \Pi^M$. Thus $J_\infty\ge J$.
		
		For the reverse inequality we start with $J_{N\pi} \le J_{\infty\pi }$ which is true  for all policies $\pi\in \wh \Pi^M$ due to the fact that $c\ge 0$. 		Taking the infimum over all policies yields $J_N\le J_\infty$ and taking the limit $N\to\infty$ we obtain $J\le J_{\infty}.$ It total, we have $J=J_\infty$.

		Let now $v\in \MM$ be another fixed point of $\T$, i.e. $v=\T v$. Iterating this equality yields $v=\T^n v$ for all $n\in\N$. Since $v\in \MM$ we have $v\ge g$ and because of the monotonicity of the Bellman operator we get $v=\T^nv\geq T^n g$. Letting $n\to\infty$ finally implies $v\geq J=J_\infty$, thus $J_\infty$ is the smallest fixed point of the Bellman operator. 
		\item Since $J_\infty \in  \MM$, the existence of a minimizing Markov decision rule follows as in the proof of  Theorem \ref{thm:finite}. Furthermore, it holds $J_\infty(x,s,t) \geq g(s), \ (x,s,t) \in \wh E$, since $J_\infty \in \MM$. Consequently, we have
		\[ J_\infty = \lim_{N \to \infty} \T_{d^*}^N J_\infty \geq \lim_{N \to \infty} \T_{d^*}^N g = \lim_{N \to \infty} J_{N\pi^*} = J_{\infty\pi^*} \geq J_\infty. \]
		i.e.\ $\pi^*$ is optimal. The first equality is by part a), the inequality thereafter by the monotonicity of the operator $\T_{d^*}$ and the second equality by the value iteration \eqref{eq:infinite_value_iteration}. \qedhere
	\end{enumerate}
\end{proof}

%--------------------------------------------------------------------------------------------------
\section{Relaxed Assumptions for Monotone Models}\label{sec:monotone}
%--------------------------------------------------------------------------------------------------

The  model has been introduced in Section \ref{sec:decision_model} with a general Borel space as state space. In order to solve the optimization problem in Sections \ref{sec:finite} and \ref{sec:infinite} we needed a continuous transition function despite having a semicontinuous model. This assumption on the transition function can be relaxed to semicontinuity if the state space is the real line and the transition and one-stage cost function have some form of monotonicity. For notational convenience, we consider the stationary model with no terminal cost under both finite and infinite horizon in this section. We replace Assumption \ref{ass:continuity_compactness} by 

\begin{assumption}\label{ass:monotone}
	\begin{enumerate}
		\item[(i)] The state space is the real line $E=\R$.
		\item[(ii)]  The sets $D(x)$ are compact and  $\R \ni x \mapsto D(x)$ is upper semicontinuous and decreasing, i.e.\ $D(x) \supseteq D(y)$ for $x \leq y$.
		\item[(iii)] The transition function $T$ is lower semicontinuous in $(x,a)$ and increasing in $x$.
		\item[(iv)] The one-stage cost $c(x,a,T(x,a,z))$ is lower semicontinuous in $(x,a)$ and increasing in $x$. 
	\end{enumerate}
\end{assumption}

Requiring that the one-stage cost function $c$ is lower semicontinuous in $(x,a,x')$ and increasing in $(x,x')$ is sufficient for Assumption \ref{ass:monotone} (iv) to hold due to part (iii) of the assumption.

How do the modified continuity assumptions affect the validity of the results in Sections \ref{sec:finite} and \ref{sec:infinite}?
The only two results that were proven using the continuity of the transition function $T$ in $(x,a)$ and not only its measurability are Theorems \ref{thm:finite} and \ref{thm:infinite}. All other statements are unaffected.

\begin{proposition}\label{thm:monotone}
	The assertions of Theorems \ref{thm:finite} and \ref{thm:infinite} hold under Assumption \ref{ass:monotone}, too. Moreover, the value functions $J_n$ and $J_\infty$ are increasing. The set of potential value functions can therefore be replaced by 
	\begin{align*}
		\MM = \big\{ v: \wh E \to \R\mid \ &  v \text{ is lower semicontinuous and increasing,}\\
%		& v(x,\cdot,\cdot) \text{ is continuous for all } x \in \R,\\
		& v(x,s,t) \ge g(s) \text{ for } (x,s,t) \in \wh E \big\}.
	\end{align*}
\end{proposition}

\begin{proof}
	In Theorem \ref{thm:finite}, the continuity of $T$ is used to show that $\wh D \ni (x,s,t,a) \mapsto Lv(x,s,t,a)$ is lower semicontinuous for every $v \in \MM$. Due to the monotonicity assumptions, the mapping 
	\[\wh D_n \ni (x,s,t,a) \mapsto v\Big(T(x,a,Z(\omega)),\, s+tc(x,a,T(x,a,Z(\omega))),\, \beta t\Big)   \]
	is lower semicontinuous for every $\omega \in \Omega$ as a composition of an increasing lower semicontinuous function with a lower semicontinuous one. Now, the lower semicontinuity of $\wh D \ni (x,s,t,a) \mapsto Lv(x,s,t,a)$ and the existence of a minimizing decision rule follow as in the proof of Theorem \ref{thm:finite}. The fact that $\T v$ is increasing for every $v \in \MM$ follows as in Theorem 2.4.14 in \citet{BaeuerleRieder2011}. In Theorem \ref{thm:infinite}, the continuity of $T$ is only used indirectly through Theorem \ref{thm:finite}. Note that $J_\infty \in \MM$ since the pointwise limit of increasing functions remains increasing. 
\end{proof}

The monotonicity properties of Assumption \ref{ass:monotone} can be used to construct a convex model.

\begin{lemma}\label{thm:abstract_total_convex}
	Let Assumption \ref{ass:monotone} be satisfied, $A$ be a subset of a real vector space, the admissible state-action-combinations $D$ be a convex set, the transition function $T$ be convex in $(x,a)$ and the one-stage cost $D \ni (x,a) \mapsto c(x,a,T(x,a,z))$ be a convex function for every $z \in \Z$. Then the value functions $J_n(\cdot,\cdot,t)$ and $J_\infty(\cdot,\cdot,t)$ are convex for every $t> 0$.   
\end{lemma}
\begin{proof}
	We prove by  induction that $J_n$ is convex in $(x,s)$ for $n\in \N_0$. Then $J_\infty$ is convex as a pointwise limit of convex functions. For $n=0$ we know that $J_0(x,s,t)= g(s)$ is convex in $(x,s)$. Now assume that $J_{n}$ is convex in $(x,s)$. Recall that $J_{n}$ increasing by Proposition \ref{thm:monotone}. Hence, for every $\omega \in \Omega$ and $t >0$ the function
	\[ (x,s,a) \mapsto J_{n}\Big(T(x,a,Z(\omega)),\, s+tc(x,a,T(x,a,Z(\omega))),\, \beta t\Big)   \]
	is convex as a composition of an increasing convex with a convex function. By the linearity of expectation 
	$(x,s,a) \mapsto  LJ_{n}(x,s,t,a)$ is convex, too, for every $t >0$. Now, the convexity of $J_n$ follows from Proposition 2.4.18 in \citet{BaeuerleRieder2011}.
\end{proof}

If $c$ is increasing in $x'$, it is sufficient to require that $c$ and $T$ are convex in $(x,a)$. The monotonicity requirements in Assumption \ref{ass:monotone} are only one option. The following alternative is relevant i.a.\ for the dynamic reinsurance model in Section \ref{sec:reinsurance}. For a proof see Section 6.1.3 in \citet{Glauner2020}.

\begin{corollary}\label{thm:decreasing}
	Change Assumption \ref{ass:monotone} (ii)-(iv) to 
	\begin{enumerate}
		\item[(ii')]  The sets $D(x)$ are compact and  $\R \ni x \mapsto D(x)$ is upper semicontinuous and increasing.
		\item[(iii')] $T$ is upper semicontinuous in $(x,a)$ and increasing in $x$.
		\item[(iv')]  $c(x,a,T(x,a,z))$ is lower semicontinuous in $(x,a)$ and decreasing in $x$. 
	\end{enumerate}
	Then, the assertions of Theorems \ref{thm:finite} and \ref{thm:infinite} still hold with the value functions $J_n$ and $J_\infty$ being decreasing in $x$ and increasing in $(s,t)$.\\
	If furthermore $A$ is a subset of a real vector space, $D$ a convex set, $T$ concave in $(x,a)$ and $D \ni (x,a) \mapsto c(x,a,T(x,a,z))$  convex for every $z \in \Z$. Then, the value functions $J_n(\cdot,\cdot,t)$ and $J_\infty(\cdot,\cdot,t)$ are convex for every $t >0$.
\end{corollary}

%--------------------------------------------------------------------------------------------------
\section{Outer Problem: Existence}\label{sec:outer_existence}
%--------------------------------------------------------------------------------------------------

In this section, we study the existence of a solution to the outer optimization problem \eqref{eq:outer_problem} under both finite and infinite planning horizon. Given a solution of the respective inner problem for every $g \in G$, the two outer problems are essentially the same and therefore treated together. We have assumed in both cases that for all $x\in E$ there exists a policy $\pi$ such that $C^{\pi x}\in L^1$ and thus $\rho_\phi(C^{\pi x})=\bar{\rho}<\infty$. Hence in what follows we can restrict to policies $\pi$ such that $\rho_\phi(C^{\pi x})\le \bar{\rho}$. In this case, we can further restrict the set $G$ in the representation of Proposition \ref{thm:spectral_inf_rep}.

\begin{lemma}\label{thm:G}
It is sufficient to consider functions $g \in G$ in the representation of Proposition \ref{thm:spectral_inf_rep} which are $\phi(1)$-Lipschitz and satisfy $0 \leq g(x) \leq \bar g(x), \ x \in \R$, where
	\begin{align*}
		\bar g(x) = \phi(1) x^+ +\bar{\rho}.
	\end{align*}
	The space of such functions is denoted by $\G$.
\end{lemma} 

\begin{proof}
	Set $C=C^{\pi x}$ to simplify the notation and assume that $\rho_\phi(C)\le \bar{\rho}$. We know from Remark \ref{rem:minimizer} that the optimal $g \in G$ corresponding to $C$ is 
	\[ g_{\phi,C}(x) = \int_0^1 \q_C(\alpha) + \frac{1}{1-\alpha}\left( x- \q_C(\alpha) \right)_+ \mu(\dif \alpha), \qquad x \in \R, \]
	with $\mu$ from Proposition \ref{thm:ES_mix}.  Since $C \geq 0$ it follows
	\begin{align*}
		g_{\phi,C}(x)
		\geq \int_0^1 \q_C(\alpha) \mu(\dif \alpha)
		\geq 0. 
	\end{align*}
	Furthermore, we have
	\begin{align*}
		g_{\phi,C}(x) &= \int_0^1 \q_C(\alpha) \mu(\dif \alpha) + \int_0^1 \frac{1}{1-\alpha}\left( x- \q_C(\alpha) \right)^+ \mu(\dif \alpha)\\
		&\leq \int_0^1 \ES_\alpha(C) \mu(\dif \alpha) + x^+ \int_0^1 \frac{1}{1-\alpha} \mu(\dif \alpha)\\
		&= \rho_{\phi}(C) + \phi(1)x^+
		\leq \bar{\rho}  + \phi(1) x^+ 		= \bar g(x).
	\end{align*}
	The first inequality uses $\q_C(\alpha) = \VaR_{\alpha}(C) \leq \ES_{\alpha}(C)$ and $C \geq 0$. The identity
	\[  \int_0^1 \frac{1}{1-\alpha} \mu(\dif \alpha) = \phi(1)\]
	is by definition of $\mu$.  As a convex function, $g_{\phi,C}$ is almost everywhere differentiable with derivative $g_{\phi,C}'(x) = \phi(F_C(x)) \leq \phi(1)$, cf.\ Remark \ref{rem:minimizer}. This establishes the Lipschitz continuity with constant $L=\phi(1)$.
\end{proof}

 For a fixed policy $\pi \in \wh \Pi^M$, the optimal solution of the outer problem is already given by Remark \ref{rem:minimizer} as  
\[ g_{\phi,C_N^{\pi x}}(s) = \int_0^1 \q_{C_N^{\pi x}}(\alpha) + \frac{1}{1-\alpha}\left( s- \q_{C_N^{\pi x}}(\alpha) \right)_+ \mu(\dif \alpha), \qquad s \in \R,\ N \in \N\cup\{\infty\}. \]
However, we solved the inner problem for arbitrary but fixed $g \in \G$. Hence, the optimal policy depends on $g$ and Proposition \ref{thm:spectral_inf_rep} is not helpful.

As a first step in ensuring the existence of a solution of the outer problem, we study the dependence of the value functions of the inner problem on $g$. In order to do so, we need some structure on $\G$.

\begin{lemma}\label{thm:G_compact}
	$(\G,m)$ is a compact metric space, where
	\[ m(g_1,g_2) = \sum_{j=1}^{\infty} 2^{-j} \frac{\max_{|s| \leq j} |g_1(s)-g_2(s)|}{1+ \max_{|s| \leq j} |g_1(s)-g_2(s)| }\] 
	is the metric of compact convergence.
\end{lemma}
\begin{proof}
	Since $\G \subseteq C(\R,\R)$, it suffices to show that $\G$ is closed w.r.t.\ $m$ and verify the assumptions of the Arzelà-Ascoli theorem. Note that convergence w.r.t.\ $m$ implies pointwise convergence. Convexity, monotonicity, the common Lipschitz constant $\phi(1)$, non-negativity and the pointwise upper bound $\bar g$ are all preserved even under pointwise convergence. Hence, $\G$ is closed w.r.t.\ $m$. Moreover, $\G$ is pointwise bounded and the common Lipschitz constant implies that it is uniformly equicontinuous.
\end{proof} 

For clarity we index the value functions with $g$. The value functions $J_0^g$ of the finite horizon inner problem and $J_\infty^g$ of the infinite horizon inner problem depend semicontinuously on $g$.

\begin{lemma}\label{thm:lsc_g_finite}
	Let Assumption \ref{ass:continuity_compactness} be satisfied. Then the functional $\G \times \wh E \ni (g,x,s,t) \mapsto J_n^g(x,s,t)$ is lower semicontinuous for all $n=0,\dots,N$.
\end{lemma}
\begin{proof}
	The proof is by backward induction. At time $N$ we have to verify that $J_N^g(x,s,t)=g(s+tc_N(x))$ is lower semicontinuous in $(g,x,s,t)$. First, note that $\G \times \R_+ \ni (g,s) \mapsto g(s)$ is continuous since if $(g_k,s_k) \to (g,s)$, then $g$ converges especially pointwise and
	\begin{align*}
		\left| g_k(s_k)-g(s) \right| &= \left| g_k(s_k)-g_k(s)+g_k(s)-g(s) \right| \leq \left| g_k(s_k)-g_k(s)\right| + \left|g_k(s)-g(s) \right|\\
		&\leq \phi(1)\left| s_k-s\right| + \left|g_k(s)-g(s) \right|  \to 0 \quad \text{as } k \to \infty.
	\end{align*}
	Now let $(g_k,x_k,s_k,t_k) \to (g,x,s,t)$ and define the increasing sequence $\{c_k\}_{k \in \N}$ trough $c_k=\inf_{\ell\geq k} c_N(x_\ell)$.\\
	\emph{Case 1:} $\{c_k\}_{k \in \N}$ is bounded above and therefore convergent with limit $\hat c$. Then
	$$\hat c= \lim_{k \to \infty} c_k = \lim_{k \to \infty} \inf_{\ell\geq k} c_N(x_\ell) = \liminf_{k \to \infty} c_N(x_k) \geq c_N(x)$$ 
	since $c_N$ is lower semicontinuous. As the functions $\{g_k\}_{k \in \N}$ and $g$ are all increasing, we get
	\begin{align*}
		\liminf_{k \to \infty} g_k(s_k+t_kc_N(x_k)) &\geq \lim_{k \to \infty} g_k(s_k+t_kc_k) = g(s+t\hat c)\geq g(s+tc_N(x)).
	\end{align*}
	\emph{Case 2:} $\{c_k\}_{k \in \N}$ is unbounded above. Then there exists $K \in \N$ such that $c_k \geq c_N(x)$ for all $k \geq K$ and
	\begin{align*}
		\liminf_{k \to \infty} g_k(s_k+t_kc_N(x_k)) &\geq \liminf_{k \to \infty} g_k(s_k+t_kc_k)\geq \lim_{k \to \infty} g_k(s_k+t_kc_N(x))= g(s+tc_N(x)).
	\end{align*}
	
	Now assume the assertion holds for $n+1$. By Theorem \ref{thm:finite} we have at time $n$
	\[ J_n^g(x,s,t) = \inf_{a \in D(x)} \E\Big[ J_{n+1}^g\Big(T_n(x,a,Z_{n+1}),\, s+tc_n(x,a,T_n(x,a,Z_{n+1})),\, \beta t\Big) \Big]. \]
	The integrand $J_{n+1}^g\Big(T_n(x,a,Z_{n+1}(\omega)),\, s+tc_n(x,a,T_n(x,a,Z_{n+1}(\omega))),\, \beta t\Big)$  is lower semicontinuous in $(g,x,s,t,a)$ for every $\omega \in \Omega$ by the induction hypothesis. Hence, if $(g_k,x_k,s_k,t_k) \to (g,x,s,t)$, Fatou's lemma and the monotonicity of expectation yield
	\begin{align*}
		&\liminf_{k \to \infty} \E\Big[ J_{n+1}^{g_k}\Big(T_n(x_k,a_k,Z_{n+1}),\, s_k+t_kc_n(x_k,a_k,T_n(x_k,a_k,Z_{n+1})),\, \beta t_k\Big) \Big]\\
		&\geq  \E\Big[\liminf_{k \to \infty} J_{n+1}^{g_k}\Big(T_n(x,a,Z_{n+1}),\, s+tc_n(x,a,T_n(x,a,Z_{n+1})),\, \beta t\Big) \Big]\\
		&\geq \E\Big[ J_{n+1}^{g}\Big(T_n(x,a,Z_{n+1}),\, s+tc_n(x,a,T_n(x,a,Z_{n+1})),\, \beta t\Big) \Big]
	\end{align*}
	I.e.\ $(g,x,s,t) \mapsto L_nJ_{n+1}^g(x,s,t,a)$ is lower semicontinuous. As the set-valued mapping $E \ni x \mapsto D(x)$ is compact valued and upper semicontinuous, 
	\[ (g,x,s,t) \mapsto J_{n}^g(x,s,t,a)= \inf_{a \in D(x)} L_nJ_{n+1}^g(x,s,t,a) \]
	is lower semicontinuous by Proposition 2.4.3 in \citet{BaeuerleRieder2011}. 
\end{proof}

\begin{lemma}\label{thm:lsc_g_infinite}
	Let Assumption \ref{ass:continuity_compactness} be satisfied. Then the functional $\G \times \wh E \ni (g,x,s,t) \mapsto J_\infty^g(x,s,t)$ is lower semicontinuous for all $(x,s,t) \in \wh E$.
\end{lemma}
\begin{proof}
	Under an infinite planning horizon, we consider a stationary model and use forward indexing for the value functions $J_N^g$. They are lower semicontinuous in $(g,x,s,t)$ by Lemma \ref{thm:lsc_g_finite}. Note that the induction basis holds especially for $c_N\equiv 0$. Since $J_N^g \uparrow J_\infty^g$ as $N \to \infty$, the assertion follows from Lemma A.1.4 in \citet{BaeuerleRieder2011}. 
\end{proof}

For initial state $x \in E$ and finite planning horizon $N \in \N$ the outer problem is given by
$\inf_{g \in \G} J_0^g(x,0,1) + \int_0^1 g^*(\phi(u)) \dif u$ and for infinite planning horizon  by  $\inf_{g \in \G} J_\infty^g(x,0,1) + \int_0^1 g^*(\phi(u)) \dif u$.
In the following, we will only use the semicontinuity of the value functions in $g$. Hence, we write 
\begin{align}\label{eq:generic_outer}
	\inf_{g \in \G} J(g) + \int_0^1 g^*(\phi(u)) \dif u
\end{align}
for a generic outer problem and suppress initial state and planning horizon.

\begin{theorem}\label{thm:outer_existence}
	Under Assumption \ref{ass:continuity_compactness} there exists a solution for the  outer optimization problem \eqref{eq:generic_outer}. 
\end{theorem}
\begin{proof}
	We want to apply  Weierstra{\ss}' extreme value theorem. In view of Lemmata \ref{thm:G_compact}, \ref{thm:lsc_g_finite} and \ref{thm:lsc_g_infinite} it suffices to show that the functional $\G \ni g \mapsto \int_0^1 g^*(\phi(u)) \dif u $
	is lower semicontinuous. Let $\{g_k\}_{k \in \N} \subseteq \G$ be a convergent sequence with limit $g \in \G$. It holds for all $u \in [0,1]$
	\begin{align}\label{eq:outer_existence_proof1}
		\liminf_{k \to \infty} g_k^*(\phi(u)) &= \lim_{k \to \infty} \inf_{\ell \geq k} g_\ell^*(\phi(u)) = \lim_{k \to \infty} \inf_{\ell \geq k} \sup_{s \in \R} \big\{ \phi(u)s -g_\ell(s) \big\}\notag\\
		&\geq \lim_{k \to \infty} \sup_{s \in \R} \inf_{\ell \geq k} \big\{ \phi(u)s -g_\ell(s) \big\} = \sup_{s \in \R} \lim_{k \to \infty} \inf_{\ell \geq k} \big\{ \phi(u)s -g_\ell(s) \big\}\notag\\
		&= \sup_{s \in \R} \big\{ \phi(u)s - \limsup_{k \to \infty} g_k(s) \big\}= \sup_{s \in \R} \big\{ \phi(u)s -  g(s) \big\} = g^*(\phi(u)).
	\end{align}
	The inequality holds generally for the interchange of infimum and supremum, the equality thereafter by Lemma A.1.6 in \citet{BaeuerleRieder2011} and the last but one equality since the sequence $\{g_k\}_{k \in \N}$ is especially pointwise convergent. Moreover note that for all $k \in \N$ and $u \in [0,1]$ it holds
	\[ g_k^*(\phi(u)) = \sup_{s \in \R} \big\{ \phi(u)s -g_k(s) \big\} \geq -g_k(0) \geq - \bar g(0) > -\infty. \]
	Now, Fatou's lemma and \eqref{eq:outer_existence_proof1} yield with
	\begin{align*}
		\liminf_{k \to \infty} \int g_k^*(\phi(u))\dif u \geq \int\liminf_{k \to \infty}  g_k^*(\phi(u))\dif u
		\geq \int g^*(\phi(u)) \dif u
	\end{align*}
	the assertion.
\end{proof}

%--------------------------------------------------------------------------------------------------
\section{Outer Problem: Numerical Approximation}\label{sec:outer_numeric}
%--------------------------------------------------------------------------------------------------

As we know now that a solution to the outer optimization problem \eqref{eq:generic_outer} exists, this section aims to determine the solution numerically. The idea is to approximate the functions $g \in \G$ by piecewise linear ones and thereby obtain a finite dimensional optimization problem which can be solved with classical methods of global optimization. We are going to show that the minimal values converge when the approximation is continuously refined and give an error bound. Regarding the second summand of the objective function \eqref{eq:generic_outer} our method coincides with the \emph{Fast Legendre-Fenchel Transform (FLT)} algorithm studied i.a.\ by \citet{Corrias1996}.

For unbounded cost $C_N^{\pi x}$ with $N \in \N \cup \{\infty\}, \ \pi \in \Pi, \ x \in E,$ the functions $g \in \G$ would have to be approximated on the whole non-negative real line. This is numerically not feasible. 

\begin{assumption}\label{ass:num}
	If $N \in \N$, we require additionally to Assumption \ref{ass:continuity_compactness} that $c$ is bounded from above by a constant $\bar c$. If $N =\infty$, we also assume that $\beta\in (0,1)$.
\end{assumption} 

Consequently, it holds $0 \leq C_N^{\pi x} \leq \hat c$ for all $N \in \N \cup \{\infty\},\ \pi \in \Pi$ and $x \in E$, where we define
\[ \hat c= \begin{cases}
	\sum_{k=0}^N \beta^k \bar c & \text{for finite planning horizon } N \in \N,\\
	\frac{\bar c}{1-\beta} & \text{for infinite planning horizon } N = \infty.
\end{cases}  \]
The bounded cost allows for a further reduction of the feasible set of the outer problem. On the reduced feasible set, the second summand of the objective function is guaranteed to be finite and easier to calculate. Recall that the convex conjugate of $g \in \G$ is an  $\bar \R$-valued function defined by
$ g^*(y) = \sup_{s \in \R} \{sy - g(s)\}, \ y \in \R. $

\begin{lemma}\phantomsection\label{thm:num_conjugate}
	\begin{enumerate}
		\item Under Assumption \ref{ass:num}, a minimizer of the outer optimization problem \eqref{eq:generic_outer} lies in 
		\[ \wh \G= \left\{  g \in \G: \ g(s)= g(0) \text{ for } s < 0 \text{ and } g(s)= g(\hat c) + \phi(1)(s-\hat c) \text{ for } s > \hat c \right\}. \]
		\item For $g \in \G$ and $y \in [0,\phi(1)]$ it holds
		$g^*(y) = \max_{s \in [0,\hat c]} \{sy-g(s)\} < \infty. $
	\end{enumerate}
\end{lemma}
\begin{proof}
	\begin{enumerate}
		\item Fix $\pi \in \Pi, x \in E$ and set $C=C_N^{\pi x}$ to simplify the notation. We know from Remark \ref{rem:minimizer} that the optimal $g \in \G$ corresponding to $C$ is 
		\[ g_{\phi,C}(s) = \int_0^1 \q_C(\alpha) + \frac{1}{1-\alpha}\left( s- \q_C(\alpha) \right)_+ \mu(\dif \alpha), \qquad s \in \R, \]
		with $\mu$ from Proposition \ref{thm:ES_mix}. Clearly, it is sufficient to consider functions $g \in G$ which are optimal for at least one $C=C_N^{\pi x}$. 
		Since $0 \leq C \leq \hat c$ we have $0 \leq \q_C(\alpha) \leq \hat c$. Consequently, it holds for $s < 0$
		\[ g_{\phi,C}(s) = \int_0^1 \q_C(\alpha) \mu(\dif \alpha) = g(0). \]
		As a convex function, $g_{\phi,C}$ is almost everywhere differentiable with derivative $g_{\phi,C}'(s) = \phi(F_C(s))$, cf.\ Remark \ref{rem:minimizer}, and for $s > \hat c$ it holds $F_C(s)=1$.
		\item Let $g \in \wh\G$ and $y \in [0,\phi(1)]$. For $s \geq \hat c$ the function
		\[ s \mapsto sy-g(s)= (y-\phi(1)) s -g(\hat c) + \phi(1)\hat c \] 
		is decreasing and for $s \leq 0$ the function
		\[ s \mapsto sy-g(s)= sy -g(0)  \] 
		is increasing. Hence, it suffices to consider the supremum over $[0,\hat c]$.\qedhere 
	\end{enumerate}
\end{proof}

The fact that the supremum of the convex conjugate reduces to the maximum of a continuous function over a compact set, opens the door for a numerical approximation with the FLT algorithm. By definition of $\wh \G$, it is sufficient to approximate the functions $g \in \wh \G$ on the interval $[0,\hat c]$. For the value iteration in Lemma \ref{thm:value_iteration} and equation \eqref{eq:infinite_value_iteration} it may be necessary to evaluate $g$ in some $s > \hat c$, but here the function is determined as a linear continuation with slope $\phi(1)$. On the interval $I=[0,\hat c]$, the metric of compact convergence reduces to the supremum norm $\|\cdot\|_\infty$. For the piecewise linear approximation we consider equidistant partitions $0=s_1<s_2<\dots<s_m=\hat c$, i.e. $s_k=(k-1) \frac{\hat c}{m-1}, \ k=1,\dots,m, \ m \geq 2$. Let us define the mapping
\[ p_m(g)(s) = g(s_k) + \frac{g(s_{k+1})-g(s_k)}{s_{k+1}-s_k} (s-s_k), \qquad s \in [s^k, s^{k+1}], \ k=1,\dots,m-1, \]
which projects a function $g \in \wh\G$ to its piecewise linear approximation and its image $\wh\G_m=\{p_m(g): \ g \in \wh\G \}$. For considering the restriction of the outer optimization problem \eqref{eq:generic_outer} to $\wh\G_m$ it is convenient to define for $g \in \wh\G$
\begin{align*}
	K_m(g) = J(p_m(g)) + \int p_m(g)^*(\phi(u)) \dif u \qquad \text{and} \qquad 
	K(g) = J(g) + \int g^*(\phi(u)) \dif u.
\end{align*}

\begin{proposition}\label{thm:num_convergence}
	It holds $$\left|\inf_{g \in \wh\G}K_m(g) -\inf_{g \in \wh\G}K(g)  \right| \leq \sup_{g \in \wh\G} |K_m(g)-K(g)| \leq 2\phi(1) \frac{\hat c}{m-1}.$$
\end{proposition}

\begin{proof}
	The first inequality is obvious and it remains to prove the second. We have for $N \in \N \cup\{\infty\}$, $x \in E$ and $g \in \wh\G$
	\begin{align*}
		\left|J(p_m(g))-J(g)\right| &= \left| \inf_{\pi \in \Pi} \E[p_m(g)(C_N^{\pi x})] - \inf_{\pi \in \Pi} \E[g(C_N^{\pi x})] \right| \leq \sup_{\pi \in \Pi} \E\left|p_m(g)(C_N^{\pi x})- g(C_N^{\pi x}) \right|\\
		&\leq \sup_{s \in I}|p_m(g)(s)-g(s)|.
	\end{align*}
	Moreover, it holds for $y \in [0,\phi(1)]$
	\begin{align*}
		|p_m(g)^*(y)-g^*(y)| &= \left| \sup_{s \in I}\{sy-g(s)\} - \sup_{s \in I} \{sy-p_m(g)(s)\} \right| \leq \sup_{s \in I} |p_m(g)(s)-g(s)|.
	\end{align*}
	Finally, the assertion follows with
	\begin{align*}
		|K_m(g)-K(g)| &\leq \left|J(p_m(g))-J(g)\right| + \int \left|  p_m(g)^*(\phi(u)) - g^*(\phi(u)) \right|\dif u \leq 2 \sup_{s \in I}|p_m(g)(s)-g(s)|\\
		&= 2\max_{k=1,\dots,m-1} \max_{s \in [s_k,s_{k+1}]} \left| g(s) -g(s_k) - \frac{g(s_{k+1})-g(s_k)}{s_{k+1}-s_k} (s-s_k) \right|\\
		&\leq 2\max_{k=1,\dots,m-1} |g(s_{k+1}) -g(s_k)| \leq 2 \phi(1) \frac{\hat c }{m-1}. \qedhere
	\end{align*}
\end{proof}

The proposition shows that the infimum of $K_m$ converges to the one of $K$. The error of restricting the outer problem \eqref{eq:generic_outer} to $\wh\G_m$ is bounded by $2\phi(1)\frac{\hat c}{m-1}$. The piecewise linear functions $g \in \wh\G_m$ are uniquely determined by their values in the kinks $s_1,\dots,s_m$. Hence, we can identify $\wh\G_m$ with the compact set
\[ \Gamma_m = \left\{ (y_1,\dots, y_m) \in \R^m: \ y_1 \in I, \ 0 \leq \frac{y_2-y_1}{s_2-s_1} \leq  \dots\leq \frac{y_m-y_{m-1}}{s_m-s_{m-1}}\leq \phi(1) \right\}. \] 
Note that due to translation invariance of $\rho_\phi$ it holds under Assumption \ref{ass:num} for $g \in \wh\G$ that $g(0)\leq \bar g(0)=\rho(\bar C)= \rho(\hat c)=\hat c$. Thus, the outer problem \eqref{eq:generic_outer} restricted to $\wh\G_m$ becomes finite dimensional:
\begin{align}\label{eq:generic_outer_findim}
	\inf_{y \in \Gamma_m} J(g_y) + \int_0^1 g_y^*(\phi(u)),
\end{align}  
where $g_y \in \wh\G_m$ is the piecewise linear function induced by $y \in \Gamma_m$, i.e.\ $$g_y(s)= y_k + \frac{y_{k+1}-y_k}{s_{k+1}-s_k}(s-s_k), \qquad s \in [s_k,s_{k+1}],\ k=1,\dots,m-1.$$

How to evaluate $J(\cdot)$ in $g_y, \ y \in \Gamma_m,$ has been discussed in Sections \ref{sec:finite} and \ref{sec:infinite}. The next Lemma simplifies the evaluation of the second summand of the objective function \eqref{eq:generic_outer_findim} to calculating the integrals $\int_{u_k}^{u_{k+1}} \phi(u) \dif u$, where $u_0=0$, $u_k= \phi^{-1}\left(\frac{y_{k+1}-y_k}{s_{k+1}-s_k}  \right),\ k=1,\dots,m-1$ and $u_m=\phi(1)$.  

\begin{lemma}\label{thm:num_conjugate_approx}
	The convex conjugate of $g_y, \ y \in \Gamma_m,$ in $\xi \in [0,\phi(1)]$ is given by
	\[ g_y^*(\xi) = \begin{cases}
		- y_1,& 0\leq \xi < \frac{y_2-y_1}{s_2-s_1},\\
		s_{k+1} \xi - y_{k+1}, & \frac{y_{k+1}-y_k}{s_{k+1}-s_k} \leq \xi \leq \frac{y_{k+2}-y_{k+1}}{s_{k+2}-s_{k+1}}, \ k=1,\dots, m-2\\
		s_m \xi -y_m, & \frac{y_m-y_{m-1}}{s_{m}-s_{m-1}} < \xi \leq \phi(1). 
	\end{cases} \]
\end{lemma}
\begin{proof}
	By Lemma \ref{thm:num_conjugate} b) we have $g_y^*(\xi) = \max_{s \in I} s\xi - g_y(s)$.  Note that the slopes $c_k= \frac{y_{k+1}-y_k}{s_{k+1}-s_k}$, $k=1,\dots, m-1$, are increasing. It follows
	\begin{align*}
		g_y^*(\xi) &= \sup_{s \in [0,\hat c]} s\xi - g_y(s)= \max_{k=1,\dots,m-1} \max_{s \in [s_k,s_{k+1}]} s (\xi-c_k) -y_k+c_ks_k.
	\end{align*}
	Let us distinguish three cases. Firstly, assume $\xi \in [c_\ell,c_{\ell+1}]$ for some $\ell \in \{1,\dots,m-2\}$. Then
	\begin{align*}
		g_y^*(\xi) &= \max \Big\{ \max_{k=1,\dots,\ell} s_{k+1}(\xi - c_k) -y_k + c_ks_k , \ \max_{k=\ell+1,\dots,m-1} s_{k}(\xi - c_k) -y_k + c_ks_k \Big\}\\
		&=	\max \Big\{ \max_{k=1,\dots,\ell} s_{k+1}\xi -y_{k+1} , \ \max_{k=\ell+1,\dots,m-1} s_{k}\xi -y_k \Big\}\\
		&= s_{\ell+1} \xi -y_{\ell+1}.
	\end{align*}
	The last equality holds, since $c_1\leq \dots \leq c_{m-1}$ and $c_{\ell} \leq \xi \leq c_{\ell+1}$ is equivalent to $\xi s_{\ell} -y_{\ell} \leq \xi s_{\ell+1} -y_{\ell+1} \geq \xi s_{\ell+2} -y_{\ell+2}$.
	Secondly, assume $\xi < c_1$. Then 
	\begin{align*}
		g_y^*(\xi) &= \max_{k=1,\dots,m-1} s_{k}(\xi - c_k) -y_k + c_ks_k
		=	\max_{k=1,\dots,m-1}  s_{k}\xi -y_k 
		= s_{1} \xi -y_{1} = -y_1.
	\end{align*}
	Again,  $\xi < c_{1}$ is equivalent to $\xi s_{2} -y_{2} < \xi s_{1} -y_{1}$.  Since $c_1\leq \dots \leq c_{m-1}$, this implies the last equality.
	The third case $c_{m-1} < \xi$ is analogous.
\end{proof}

The results of this section can be used to set up an algorithm for the problems in \eqref{eq:opt_crit_finite} and \eqref{eq:opt_crit_infinite}. First we have to set $m= \left\lceil \frac{2\phi(1)\hat c }{\epsilon}\right\rceil +1$ when we want to solve the problem with error estimate $\epsilon$. Then choose $y_0\in \Gamma_m$ and solve the inner problem with $g_{y_0}$. Use a global optimization procedure to determine the next $y_1$ (note that we do not have convexity of \eqref{eq:generic_outer_findim} in $y$) like e.g. simulated annealing to determine the optimal value of \eqref{eq:generic_outer_findim} . 

%\begin{algorithm}[H]
%	\DontPrintSemicolon
%	\caption{Outer problem}
%	\KwData{Markov Decision Model}
%	\KwResult{Optimal policy $\pi^*$, minimal risk-sensitive cost $\rho_\phi(C_N^{\pi^* x})$}
%	1. Select an approximation error $\epsilon>0$ and set $m= \left\lceil \frac{2\phi(1)\hat c }{\epsilon}\right\rceil +1$.\;
%	2. Solve \eqref{eq:generic_outer_findim} with an algorithm for global optimization.\;
%	\Indp\Indp
%	\lIf{$N \in \N$}{For each evaluation of $J(\cdot)$ solve the inner problem \eqref{eq:inner_problem_finite} with Theorem \ref{thm:finite}.}
%	\lIf{$N=\infty$}{For each evaluation of $J(\cdot)$ solve the inner problem \eqref{eq:inner_problem_infinite} with Theorem \ref{thm:infinite}.}
%	\Indm\Indm
%\end{algorithm}

%--------------------------------------------------------------------------------------------------
\section{Dynamic Optimal Reinsurance}\label{sec:reinsurance}
%--------------------------------------------------------------------------------------------------

As an application, we present a dynamic extension in discrete time of the static optimal reinsurance problem
\begin{align}\label{eq:classical_reinsurance_problem}
	\min_{f \in \F} \quad r_{\text{CoC}} \cdot \rho \big(f(Y) + \pi_R(f)\big),
\end{align}
In this setting, an insurance company incurs an aggregate loss $Y \in L^1_{\geq 0}$  at the end of a fixed period due to insurance claims. In order to reduce its risk, the insurer may cede a portion of it to a reinsurance company and retain only $f(Y)$. Here, the reinsurance treaty $f$ determines the retained loss $f(Y(\omega))$ in each scenario $\omega \in \Omega$. For the risk transfer, the insurer has to compensate the reinsurer with a reinsurance premium $\pi_R(f)= \pi_R(Y-f(Y))$, where $\pi_R:L^1_{\geq 0} \to \bar \R$ is a premium principle with properties similar to a risk measure. Most widely used is the expected premium principle $\pi_R(X)=(1+\theta)\E[X]$ with safety loading $\theta >0$. In order to preclude moral hazard, it is standard in the actuarial literature to assume that both $f$ and the ceded loss function $\id_{\R_+} -f$ are increasing. Hence, the set of admissible retained loss functions is
\[ \F= \{f: \R_+ \to \R_+ \mid f(t) \leq t \ \forall t \in \R_+, \ f \text{ increasing}, \ \id_{\R_+}-f \text{ increasing} \}. \]
The insurer's target is to minimize its cost of solvency capital which is calculated as the cost of capital rate $r_{\text{CoC}} \in (0,1]$ times the solvency capital requirement determined by applying the risk measure $\rho$ to the insurer's effective risk after reinsurance.

First research on the optimal reinsurance problem \eqref{eq:classical_reinsurance_problem} dates back to the 1960s. \citet{Borch1960} proved that a stop-loss reinsurance treaty minimizes the variance of the retained loss of the insurer given the premium is calculated with the expected value principle. A similar result has been derived in \citet{Arrow1963} where the expected utility of terminal wealth of the insurer has been maximized. Since then a lot of generalizations of this problem have been considered. For a comprehensive literature overview, we refer to \citet{Albrecher2017}. Since the 2000s, Expected Shortfall has become of special interest. \citet{ChiTan2013} identified layer reinsurance contracts as optimal for Expected Shortfall under general premium principles. Their results were extended to general distortion risk measures by \citet{Cui2013}. Other generalizations concerned additional constraints, see e.g.\ \citet{Lo2017}, or multidimensional settings induced by a macroeconomic perspective, see \citet{BauerleGlauner2018}. We are not aware of any dynamic generalizations in the literature.

Reinsurance treaties are typically written for one year, cf.\ \citet{Albrecher2017}. Hence, it is appropriate to model such an extension in discrete time. The insurer's annual surplus has the dynamics
\[ X_0=x, \qquad X_{n+1}= X_n+Z_{n+1}-f_n(Y_{n+1})-\pi_R(f_n), \] 
where the bounded, non-negative random variable $Z_{n+1} \in L^\infty_{\geq 0}$ represents the insurer's premium income in the $n$-th period. The premium principle $\pi_R:L^p_{\geq 0} \to \bar \R$ of the reinsurer is assumed to be law-invariant, monotone, normalized and to have the Fatou property. Normalization means that $\pi_R(0)=0$ and the Fatou property is lower semicontinuity w.r.t.\ dominated convergence. 

The Markov Decision Model is given by the state space $E=\R$, the action space $A=\F$, either no constraint or a budget constraint $D(x) = \{f \in \F: \pi_R(f) \leq x^+ \}$, the independent disturbances $(Y_n,Z_n)_{n \in \N}$ with $Y_n \in L^1_{\geq 0}$ and $Z_n \in L^\infty_{\geq 0}$, the transition function $T(x,f,y,z) = x - f(y) - \pi_R(f) + z$ and the one-stage cost function $c(x,f,x')= x-x'$. The insurance companies target is to minimize its solvency cost of capital for the total discounted loss
\begin{align}\label{eq:cost_of_capital}
	\inf_{\pi \in \Pi} \quad r_{\text{CoC}} \cdot \rho_{\phi}\left( \sum_{k=0}^{N-1} \beta^k \Big(  d_k(H_k^\pi)(Y_{k+1}) + \pi_R(d_k(H_k^\pi)) - Z_{k+1} \Big) \right)
\end{align}
where $\rho_\phi$ is a spectral risk measure with bounded spectrum $\phi$, $\beta\in (0,1]$ and $N \in \N$.  As it is irrelevant for the minimization, we will in the sequel omit the cost of capital rate $r_{\text{CoC}}$ and instead minimize the capital requirement. For $\beta = 1$ we have
\[ \sum_{k=0}^{N-1}  d_k(H^\pi_k)(Y_{k+1}) + \pi_R(d_k(H^\pi_k)) - Z_{k+1} = \sum_{k=0}^{N-1} X_k^\pi - X_{k+1}^\pi =  x-X_N^\pi, \]
i.e.\ due to translation invariance of spectral risk measures the objective reduces to minimizing the capital requirement for the loss (negative surplus) at the planing horizon $-X_N^\pi$. This is reminiscent of the static reinsurance problem \eqref{eq:classical_reinsurance_problem}, however here the loss distribution at the planing horizon can be controlled by interim action. Throughout, we have required that the one-stage cost $c(x,f,T(x,f,Y,Z))= f(Y)+ \pi_R(f) -Z$ is non-negative. As $f(Y)$ and $ \pi_R(f)$ are non-negative for all $f \in \F$ and $c(x,\id_{\R_+},T(x,\id_{\R_+},Y,Z))=Y-Z$ due to normalization of $\pi_R$, the premium income $Z$ would have to be non-positive. This makes no sense from an actuarial point of view, but since $\rho_\phi$ is translation invariant and $Z \in L^{\infty}$ we can add $\sum_{k=0}^{N-1} \beta^k \esssup(Z)$ without influencing the minimization. This means that the one-stage cost function is changed to $\hat c(x,f,x')= x-x'+\esssup(Z)$. The economic interpretation is that the one-stage cost 
\[ \hat c(x,f,T(x,f,Y,Z)) = f(Y)+ \pi_R(f)+ \esssup(Z) -Z \]
now depends on the deviation from the maximal possible income instead of the actual income.  For brevity we write $\hat z= \esssup(Z)$.

As in \eqref{eq:outer_problem} we separate an inner and outer reinsurance problem. For a structural analysis we focus on the inner optimization problem
\begin{align}\label{eq:cost_of_capital_inner}
	\inf_{\pi \in \Pi} \E\left[g\left( \sum_{k=0}^{N-1} \beta^k \Big(  d_k(H_k^\pi)(Y_{k+1})
	+ \pi_R(d_k(H_k^\pi))+ \hat z -  Z_{k+1} \Big) \right)\right]
\end{align}
with arbitrary $g \in \G$, cf.\ Lemma \ref{thm:G}. Note that for $\pi=(f,f,\ldots)$ with $f=\id_{\R_+}$ we obtain $\rho_\phi(C_N^{\pi x})<\infty$.  On the extended state space $\wh E= \R \times \R_+\times (0,1]$, the value of a policy $\pi \in \wh \Pi$ is defined as
\begin{align*}
	V_{N\pi}(h_N) &= g(s_N),\\
	V_{n\pi}(h_n) &= \E_{nh_n}\left[g\left(s_n + t_n\sum_{k=n}^{N-1} \beta^{k-n} \Big(  d_k(H_k^\pi)(Y_{k+1})
	+ \pi_R(d_k(H_k^\pi))+ \hat z - Z_{k+1} \Big)\right)\right],
\end{align*}
for $n=0,\dots,N$ and $h_n \in \wh \H_n$. The corresponding value functions are
\begin{align*}
	V_{n}(h_n) = \inf_{\pi \in \wh\Pi} V_{n\pi}(h_n), \qquad h_n \in \wh\H_n.
\end{align*}
Due to the real state space we want to apply Corollary \ref{thm:decreasing} for solving the optimization problem. Note that  the one-stage cost $\hat c$ is non-negative and the spectrum $\phi$ bounded by assumption. 
The following lemma shows that also the monotonicity, continuity and compactness assumptions of Corollary \ref{thm:decreasing} are satisfied by the dynamic reinsurance model.

\begin{lemma}\phantomsection\label{thm:retained_loss_functions}
	\begin{enumerate}
		\item The retained loss functions $f \in \F$ are Lipschitz continuous with constant $L\leq1$. Moreover, $\F$ is a Borel space as a compact subset of the metric space $(C(\R_+),m)$ of continuous real-valued functions on $\R_+$ with the metric of compact convergence.
		\item The functional $\pi_R:\F \to \R_+, \ f \mapsto \pi_R(f)$ is lower semicontinuous.
		\item The transition function $T$ is upper semicontinuous and increasing in $x$.
		\item $D(x)$ is a compact subset of $\F$ for all $x \in \R$ and the set-valued mapping $\R \ni x \mapsto D(x)$ is upper semicontinuous and increasing.
		\item The one-stage cost $D \ni (x,a)\mapsto c(x,a,T(x,a,y,z))$ is lower semicontinuous and decreasing in $x$.
	\end{enumerate}
\end{lemma}

\begin{proof}
	\begin{enumerate}
		\item Let $f \in \F$. Since $\id_{\R_+}-f$ is increasing, it holds for $0\leq x\leq y$ that $x-f(x) \leq y-f(y)$. Rearranging and using that $f$ is increasing, too, yields with $|f(x)-f(y)|=f(y)-f(x) \leq y-x = |x-y|$ the Lipschitz continuity with common constant $1$. Moreover, $\F$ is pointwise bounded by $\id_{\R_+}$ and closed under pointwise convergence. Hence, $(\F,m)$ is a compact metric space by the Arzelà-Ascoli theorem and as such also complete and separable, i.e. a Borel space.
		\item Let $\{f_k\}_{k \in \N}$ be a sequence in $\F$ such that $f_k \to f \in \F$. Especially, it holds $f_k(x) \to f(x)$ for all $x \in \R_+$ and $Y-f_k(Y) \to Y-f(Y)$ $\P$-a.s. Since $Y-f_k(Y) \leq Y \in L^1$ for all $k \in \N$, the Fatou property of $\pi_R$ implies
		\[ \liminf_{k \to \infty} \pi_R(f_k) = \liminf_{k \to \infty} \pi_R\big(Y-f_k(Y)\big) \geq \pi_R\big(Y-f(Y)\big)=\pi_R(f). \]
		\item We show that the mapping $\F \times \R_+ \ni (f,y) \mapsto f(y)$ is continuous. Then, the transition function $T$ is upper semicontinuous as a sum of upper semicontinuous functions due to part b). Let $\{(f_k,y_k)\}_{k \in \N}$ be a convergent sequence in $\F \times \R_+$ with limit $(f,y)$. Since convergence w.r.t.\ the metric $m$ implies pointwise convergence and all $f_k$ have the Lipschitz constant $L=1$, it follows
		\begin{align*}
			\left| f_k(y_k)-f(y) \right| \leq \left| f_k(y_k)-f_k(y)\right| + \left|f_k(y)-f(y) \right|\leq \left| y_k-y\right| + \left|f_k(y)-f(y) \right|  \to 0.
		\end{align*}
		The fact that $T$ is increasing in $x$ is obvious.
		\item Due to a), we only have to consider the budget-constrained case. Since $\F$ is compact it suffices to show that $D(x) = \{f \in \F: \pi_R(f) \leq x^+ \}$ is closed. This is the case since $D(x)$ is a sublevel set of the lower semicontinuous function $\pi_R:\F \to \R_+$, cf. Lemma A.1.3 in \citet{BaeuerleRieder2011}. Furthermore, we show that $D$ is closed to obtain the upper semicontinuity from Lemma A.2.2 in \citet{BaeuerleRieder2011}. From the lower semicontinuity of $\pi_R$ it follows that the epigraph 
		\[ \epi(\pi_R)= \{ (f,x) \in \F\times \R_+: \pi_R(f) \leq x \} \]
		is closed. Thus, $D=\{(x,f): (f,x) \in \epi(\pi_R)\} \cup ( \R_- \times D(0))$ is closed, too. That $x \mapsto D(x)$ is increasing is clear.
		\item The one-stage cost $c(x,f,T(x,a,y,z)) = x - T(x,f,y,z)=f(y)+\pi_R(f)-z$ is lower semicontinuous in $(x,f)$ as a sum of lower semicontinuous functions and decreasing in $x$ since it does not depend on $x$.  \qedhere
	\end{enumerate}
\end{proof}

Now, Corollary \ref{thm:decreasing} yields that it is sufficient to minimize over all Markov policies and the value functions satisfy the Bellman equation
\begin{align*}
	J_N(x,s,t)&= g(s),\\
	J_n(x,s,t)&= \inf_{f \in D(x)} \E\Big[ J_{n+1}\Big(x - f(Y) - \pi_R(f) + Z,\, s+t\big(f(Y)+\pi_R(f)+\hat z- Z\big),\, \beta t\Big) \Big]
\end{align*} 
for $(x,s,t) \in \wh E$ and $n=0,\dots,N-1$. Moreover, there exists a Markov Decision rule $d_n^*: \wh E \to \F$ minimizing $J_{n+1}$ and every sequence $\pi=(d_0^*,\dots,d_{N-1}^*) \in \wh\Pi^M$ of such minimizers is at solution to \eqref{eq:cost_of_capital_inner}.

All structural properties of the optimal policy which do not depend on $g$ are inherited by the optimal solution of the cost of capital minimization problem \eqref{eq:cost_of_capital}. The structural properties we will focus on in the rest of this section are induced by convexity. Therefore, we assume that the premium principle $\pi_R$ is convex and that there is no budget constraint. Note that $D(x)$ is non-convex even for convex $\pi_R$. Then, we have indeed a convex model: $D$ is trivially convex, the transition function   
$T(x,f,y,z) = x - f(y) - \pi_R(f) + z$ is concave in $(x,f)$ as a sum of concave functions and the one-stage cost $(x,f) \mapsto \hat c(x,f,T(x,f,y,z)) = f(y)+\pi_R(f)+\hat z- z$
is convex as a sum of convex functions. Now, Corollary \ref{thm:decreasing} yields that the value functions $J_n$ are convex. Under the widely-used expected premium principle, the optimization problem can be reduced to finite dimension.

\begin{example}
	Let $\pi_R(\cdot) = (1+\theta)\E[\cdot]$ be the expected premium principle with safety loading $\theta >0$ and assume there is no budget constraint. We will now show that the optimal reinsurance treaties (i.e.\ retained loss functions) can be chosen from the class of \emph{stop loss} treaties
	\[ f(x) = \min\{x,a\}, \qquad a \in [0,\infty]. \]
	Due to the convexity of $J_{n+1}$, we can infer from the Bellman equation that reinsurance treaty $f_1$ is better than $f_2$ if  
	\begin{align*}
		f_1(Y)+ \pi_R(f_1)\leq_{cx} f_2(Y)+ \pi_R(f_2),
	\end{align*}
	where $\leq_{cx}$ denotes the convex order. Since $Y_1 \leq_{cx} Y_2$ implies $\E[Y_1]=\E[Y_2]$, it suffices to find an $a_f \in [0,\infty]$ such that 
	\begin{align}\label{eq:cost_of_capital_cx_example}
		\min\{Y,a_f\} \leq_{cx} f(Y).
	\end{align}
	The mapping $[0,\infty] \to \R_+, \ a \mapsto \min\{Y(\omega),a\}$ is continuous for all $\omega \in \Omega$ and $0 \leq  \min\{Y,a\} \leq Y \in L^1$. Thus, it follows from dominated convergence that $[0,\infty] \to \R_+, \ a \mapsto \E[\min\{Y,a\}]$ is continuous. Furthermore,
	\[ \E[\min\{Y,0\}] \leq \E[f(Y)] \leq \E[\min\{Y,\esssup(Y)\}]. \]
	Hence, by the intermediate value theorem there is an $a_f \in [0,\infty]$ such that $\E[f(Y)] = \E[\min\{Y,a_f\}]$. Let us compare the survival functions:
	\begin{align*}
		S_{\min\{Y,a_f\}}(y)&=\P(\min\{Y,a_f\}>y)=\P(Y>y)\1\{a_f>y\},\\
		S_{f(Y)}(y)&= \P(f(Y)>y)\leq \P(Y>y).
	\end{align*}
	The inequality holds since $f \leq \id_{\R_+}$. Hence, we have $S_{\min\{Y,a_f\}}(y) \geq S_{f(Y)}(y)$ for $y<a_f$ and $S_{\min\{Y,a_f\}}(y) \leq S_{f(Y)}(y)$ for $y\geq a_f$. The cut criterion 1.5.17 in \citet{MuellerStoyan2002} implies $\min\{Y,a_f\} \leq_{icx} f(Y)$ and due to the equality in expectation follows \eqref{eq:cost_of_capital_cx_example}, cf.\ Theorem 1.5.3 in \citet{MuellerStoyan2002}. So the inner optimization problem \eqref{eq:cost_of_capital_inner} is reduced to finding an optimal nonnegative parameter of a stop loss treaty at every stage.
\end{example}

\bibliographystyle{apalike}
\renewcommand{\bibfont}{\small}
\bibliography{Literature_MDP_spectral_risk_measure}

\end{document}